\documentclass{amsart}[12pt]

\usepackage{mystyle}

\title{\Large The upsilon invariant at $1$ 
of $3$-braid knots} 

\author{Paula Truöl}
\address{Department of Mathematics, ETH Zurich, Switzerland}
\email{paula.truoel@math.ethz.ch}
\begin{document}

\def\subjclassname{\textup{2020} Mathematics Subject Classification}
\expandafter\let\csname subjclassname@1991\endcsname=\subjclassname

\begin{abstract}
We provide explicit formulas for the integer-valued smooth concordance invariant $\upsilon(K) = \Upsilon_K(1)$ for every $3$-braid knot $K$. We determine this invariant, which was defined by Ozsv\'ath, Stipsicz and Szab\'o \cite{OSS2017}, by constructing cobordisms between $3$-braid knots and (connected sums of) torus knots. 
As an application, we show that for positive $3$-braid knots $K$ several alternating distances all equal the sum $g(K) + \upsilon(K)$, where $g(K)$ denotes the $3$-genus of $K$.
In particular, we compute the alternation number, the dealternating number and the Turaev genus for all positive $3$-braid knots. We also provide upper and lower bounds on the alternation number and dealternating number of every $3$-braid knot which differ by $1$.
\end{abstract}

\maketitle

\section{Introduction}

We study \emph{knots} in the $3$-sphere $S^3$, \ie non-empty, connected, oriented, closed smooth $1$-dimensional submanifolds of $S^3$, considered up to ambient isotopy.
Two knots $K$ and $J$ are called \emph{concordant} 
if there exists an annulus $A \cong S^1 \times [0,1] $ smoothly and properly embedded in $S^3 \times [0,1]$ such that $\partial A = K \times \{0\} \,\cup J \,\times \{1\}$ and such that the induced orientation on the boundary of the annulus agrees with the orientation of $K$, but is the opposite one on $J$.
Knots up to concordance form a group, the \emph{concordance group} $\CC$, with the group operation induced by connected sum. \\

In \cite{OSS2017}, Ozsv\'ath, Stipsicz and Szab\'o used the Heegaard Floer knot complex to define the invariant $\Upsilon_K$ of a knot $K$,
which induces a homomorphism from the knot concordance group to the group of real-valued piecewise linear functions on the interval $[0,2]$. 
The function $\Upsilon_K$ evaluated at $t=1$, $\upsilon(K):=\Upsilon_K(1)$, induces a homomorphism $\CC \to \Z$. In this article, we will call $\upsilon(K)$ \emph{upsilon of $K$}.\\

A $3$-\emph{braid} is an element of the \emph{braid group on three strands}, denoted $B_3$. The classical presentation of $B_3$ with generators $a$ and $b$ and relation $aba=bab$, the \emph{braid relation},
was introduced by Artin \cite{artin_1925}. A \emph{braid word} $\gamma$ --- a word in the generators of $B_3$ and their inverses --- defines a diagram for a (geometric) $3$-braid;
the generators $a$ and $b$ correspond to the geometric $3$-braids given by braid diagrams as in \Cref{fig:braidgroup}. In our figures, braid diagrams will always be oriented from bottom to top.
We denote by $\Delta$ the braid $aba = bab$, and note that its square $\Delta^2=(ab)^3$ (the positive full twist on three strands) generates the center of $B_3$ 
\cite[Theorem 3]{chow}. 
A \emph{$3$-braid knot} is a knot that arises as the closure $\widehat{\gamma}$ of a $3$-braid $\gamma$.\\

\captionsetup[subfigure]{labelfont=normalfont, labelformat = simple}
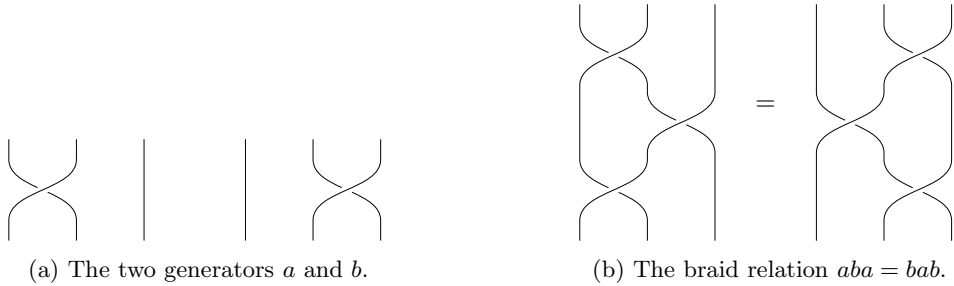
\begin{figure}[htbp]
     \centering
     \begin{subfigure}[t]{0.4\textwidth}
         \centering
\resizebox{\textwidth}{!}{%
         \begin{tikzpicture}
\pic[braid/number of strands=3] at (0,0) {braid=a_1^{-1}};
\pic[braid/number of strands=3] at (3.5,0) {braid=a_2^{-1}};
\end{tikzpicture}
}
         \caption{The two generators $a$ and $b$.}
     \end{subfigure}
     \hfill
     \begin{subfigure}[t]{0.4\textwidth}
         \centering
         \resizebox{\textwidth}{!}{%
         \begin{tikzpicture}
\pic[braid/number of strands=3] at (0,0) {braid=a_1^{-1} a_2^{-1} a_1^{-1}
};
\node[font=\large] at (2.75,-1.5) {\(=\)};
\pic[braid/number of strands=3] at (3.5,0) {braid=a_2^{-1} a_1^{-1} a_2^{-1}};
\end{tikzpicture}
}
         \caption{The braid relation $aba=bab$.}
     \end{subfigure}
        \caption{Generators and relation in the braid group $B_3$.}
        \label{fig:braidgroup}
\end{figure}

As our main result, we determine the upsilon invariant for all $3$-braid knots. More precisely, we show the following.

\begin{theorem}\label{thm:upsilon}
Let $\gamma = \Delta^{2\ell}a^{-p_1}b^{q_1}a^{-p_2}b^{q_2}\cdots a^{-p_r}b^{q_r}$ be a braid word in the generators $a$ and $b$ of $B_3$
for some integers $\ell\in \Z$, $r\geq 1$ and $p_i,q_i\geq 1$ for $i\in \{1, \dots, r\}$, where $\Delta^2=(ab)^3$.
Suppose that the closure $K=\widehat{\gamma}$ of $\gamma$ is a knot. Then its upsilon invariant is
\begin{align*}
\upsilon(K)&=\dfrac{\sum\limits_{i=1}^r (p_i - q_i) }{2}-2\ell.
\end{align*}
\end{theorem}

It follows from Murasugi's classification of the conjugacy classes of $3$-braids \cite[Proposition 2.1]{murasugibook} that indeed all $3$-braid knots --- except for the torus knots that are closures of $3$-braids --- are covered by \Cref{thm:upsilon}. 
However, for torus knots the invariant $\upsilon$ can be calculated explicitly by a combinatorial, inductive formula in terms of their Alexander polynomial \cite[Theorem 1.15]{OSS2017}; see \Cref{lem:Upsilontorus} below.
Hence, we have indeed determined $\upsilon(K)$ for all $3$-braid knots $K$.
\\

As an application of \Cref{thm:upsilon}, we show that the following invariants coincide for \emph{positive $3$-braid knots} --- knots that are the closure of positive $3$-braids.

\begin{cor}\label{cor:alt}
Let $K$ be a knot that is the closure of a \emph{positive $3$-braid}, i.e.~an element of $B_3$ that can be written as a word in the generators $a$ and $b$ only (no inverses). Then
\begin{align*}
\alt(K) &= \dalt(K) = g_T(K) = \mathcal{A}_s(K)
= g(K) +\upsilon(K) .
\end{align*}
\end{cor}

Here, the \emph{alternation number} $\alt(K)$, \emph{dealternating number} $\dalt(K)$ and \emph{Turaev genus} $g_T(K)$ are different ways of measuring 
how far the knot $K$ is from being alternating. 
The best known among them is certainly the first one: the alternation number $\alt(K)$ of a knot $K$ was first defined by Kawauchi \cite{kawauchi} as the minimal \emph{Gordian distance} of $K$ to the set of alternating knots. In \Cref{sec:alternation}, we will review the precise definition and prove \Cref{cor:alt}.
The invariant $\mathcal{A}_s(K)$ introduced by Friedl, Livingston and Zentner \cite{friedllivingstonzentner} is defined as the minimal number of double point singularities in a generically immersed concordance from a knot $K$ to an alternating knot. Lastly, $g(K)$ denotes the \emph{$3$-genus} of $K$, the minimal genus of a compact, connected, oriented smooth surface in $S^3$ with oriented boundary the knot $K$.\\

Two other corollaries of \Cref{thm:upsilon} for positive $3$-braid knots are the following.

\begin{cor}\label{cor:switches}
Let $K$ be a positive $3$-braid knot.
Then the minimal $r$ such that $K$ is the closure of $a^{p_1}b^{q_1}a^{p_2}b^{q_2}\cdots a^{p_r}b^{q_r}$ for positive integers $p_i,q_i$, $i \in \{1, \dots, r\}$, is $r = g(K)+\upsilon(K)  +1$.
\end{cor}

\begin{cor}\label{cor:conc}
If $K$ and $J$ are concordant knots that are both closures of positive $3$-braids, then
the minimal $r$ from \Cref{cor:switches} is the same for both $K$ and $J$.
\end{cor}

\Cref{prop:normalform} in \Cref{sec:normalform} provides a normal form for $3$-braids, the \emph{Garside normal form}, which is different from Murasugi's normal form mentioned above (cf.~\Cref{def:murasuginormalform}). 
The Garside normal form allows us to read off from a braid word whether it is conjugate to a positive braid word.
In \Cref{sec:homogenization}, we provide formulas for the fractional Dehn twist coefficient for all $3$-braids in Garside normal form; see 
\Cref{cor:FDTC}.\\

\textbf{Proof strategy for \Cref{thm:upsilon}.} 
A crucial property of the invariant $\upsilon$ is that it provides a lower bound on the \emph{$4$-genus} $g_4(K)$ of a knot $K$, the minimal genus of a compact, connected, oriented surface smoothly embedded in the $4$-ball $B^4$ with oriented boundary the knot $K$ in $S^3 = \partial B^4$: we have
\begin{align}\label{eq:Upsilon4genus}
\left|\upsilon(K)\right|\leq g_4(K) 
\end{align}
for any knot $K$ \cite[Theorem 1.11]{OSS2017}.
Our general strategy to find $\upsilon(K)$ for any $3$-braid knot $K$ 
will be to construct a cobordism between $K$ and another knot $J$ for which the value of $\upsilon$ is known.
A \emph{cobordism} between $K$ and $J$ is a 
smoothly and properly embedded oriented surface $C$ in $S^3 \times [0,1]$ with boundary $K \times \{0\} \cup J \times \{1\}$ such that the induced orientation on the boundary of $C$ agrees with the orientation of $K$ and disagrees with the orientation of $J$. 
We have 
\begin{align}\label{eq:upsilonboundintro}
\left\vert\upsilon(K)-\upsilon(J)\right\vert 
\leq g(C),
\end{align}
for any cobordism $C$ between $K$ and $J$, where
$g(C)$ denotes the genus of the cobordism; see inequality \eqref{eq:strategy} in \Cref{sec:methodology}. This provides bounds on $\upsilon(K)$ in terms of $\upsilon(J)$ and $g(C)$.

We will find such cobordisms for example by algebraic modifications of a braid word representing $K$ and by saddle moves corresponding to the addition or deletion of generators from such braid words. We will also repeatedly make use of the trick described in \Cref{ex:introstrategy} in \Cref{sec:methodology} of looking at cobordisms of genus $1$ between $\widehat{\gamma}\# T_{2,2n+1} $ and $\widehat{\gamma b^{2n}}$ for $3$-braid words $\gamma$ and $n \geq 1$.\\

To prove \Cref{thm:upsilon}, we will first determine $\upsilon$ for all positive $3$-braid knots and then generalize our computations to all $3$-braid knots. This extension was somewhat unexpected for the author
since, in contrast, the same method would not work to determine slice-torus invariants \cite{lewark} like the invariant $\tau$ defined by Ozsv\'ath and Szab\'o \cite{Ozsv_th_2003} or Rasmussen's invariant $s$ \cite{Rasmussen} for all $3$-braid knots. We will elaborate on this in \Cref{sec:prooftechnique}.

\begin{rem}
As we will only use properties of the upsilon invariant (see \Cref{sec:upsilonprops}) and not its definition, we can similarly determine any concordance homomorphism $\CC \to \mathbb{Z}$ whose absolute value bounds the $4$-genus of a knot from below and which takes the same value as $\upsilon$ on torus knots of braid index $2$ and $3$.
An example is $-\frac{t}{2}$ for the concordance invariant $t$ constructed by Ballinger \cite{ballinger} from the $E(-1)$ spectral sequence on Khovanov homology. 
The invariant $t$ defines a concordance homomorphism valued in the even integers which satisfies $\left|\frac{t(K)}{2}\right|\leq g_4(K)$ for any knot $K$ \cite[Theorem 1.1]{ballinger}. Moreover, it fulfills $t\left(T_{p,q}\right) = -2\upsilon \left(T_{p,q}\right)$ for the torus knots $T_{p,q}$ for any coprime positive integers $p$ and $q$ \cite[p.~22]{ballinger}. 
The same method we use for the proof of \Cref{thm:upsilon} shows that $t(K) = -2\upsilon(K)$ for any $3$-braid knot $K$.
\end{rem}

\begin{rem}
\Cref{thm:upsilon} and a result of Erle \cite{erle} imply
that $\sigma(K) = 2\upsilon(K)$ for all $3$-braid knots $K$ except when 
$K=\pm T_{3,3\ell+k}$ for odd $\ell >0$ and $k \in \{1,2\}$.
Here $\sigma(K)$ denotes the classical signature of the knot $K$ \cite{trotter}\footnote{We use the standard signature convention that the positive torus knots have negative signatures, \eg $\sigma(T_{3,2}) = -2$.}.
In the exceptional cases, we have $\sigma(K) = 2\upsilon(K)-2$. This observation improves a result by Feller and Krcatovich who showed that $\left \vert \upsilon(K) - \frac{\sigma(K)}{2}\right \vert \leq 2$ for all $3$-braid knots $K$ \cite[Proposition 4.4]{Feller_2017}; see also \Cref{rem:upsilonsignature}.
\end{rem}

\textbf{Organization.}
The remainder of this article is organized as follows. 
In \Cref{sec:prelims}, we will provide the necessary background on (positive) braids and the upsilon invariant before providing a normal form for $3$-braids (\Cref{prop:normalform}) that we call Garside normal form in \Cref{sec:normalform}.
Then in \Cref{sec:Upsilon}, after a more detailed outline of our proof strategy (\Cref{sec:methodology}), we will prove \Cref{thm:upsilon} 
first for positive $3$-braid knots (\Cref{sec:upsilonpos}) and afterwards in the general $3$-braid case (\Cref{sec:upsilongeneral}). We will prove \Cref{cor:switches} and \Cref{cor:conc} 
in \Cref{sec:upsilonpos}.
\Cref{sec:discussion} will provide further context on our results. 
\Cref{sec:alternation} is concerned with the proof of \Cref{cor:alt} (\Cref{sec:altpos}) and the application of our result about the upsilon invariant to alternating distances of general $3$-braid knots (\Cref{sec:altgeneral}). In particular, we determine the alternation number of any $3$-braid knot up to an additive error of at most $1$. 
Finally, in Section \ref{sec:homogenization}, we determine the fractional Dehn twist coefficient for all $3$-braids in Garside normal form.\\

\textbf{Acknowledgments.}
I would like to thank Peter Feller for introducing me to the topic and for all the helpful discussions. Thanks also to Lukas Lewark for lots of useful comments, including during my stay in Regensburg in September 2020, and to Xenia Flamm for her feedback. Finally, I thank the referee for many valuable remarks and improvements.
This project is supported by the Swiss National Science Foundation Grant 181199.

\section{Preliminaries}\label{sec:prelims}

We recall important concepts about knots and braids, and also the necessary properties of the upsilon invariant and the knot invariant $\tau$ coming from Heegaard Floer homology.

\subsection{Knots and braids}\label{sec:braids}

By a fundamental theorem of Alexander \cite{alexander}, every knot in $S^3$ can be represented as the closure of a geometric $n$-braid for some positive integer $n$. 
An $n$-\emph{braid} is an element of the \emph{braid group on $n$ strands}, denoted $B_n$, which is presented by $n-1$ generators $\sigma_1, \dots, \sigma_{n-1}$ and relations
\begin{align*}
\sigma_i \sigma_j = \sigma_j \sigma_i \text{ if } | i-j| \geq 2, \text{ and }  \sigma_i \sigma_{i+1} \sigma_i = \sigma_{i+1} \sigma_{i} \sigma_{i+1} \qquad  \text{\cite{artin_1925}}.
\end{align*}
We call a word in the generators of $B_n$ and their inverses a \emph{braid word}. A braid word defines a diagram for a (geometric) $n$-braid where the generators $\sigma_i$ of the braid group correspond to the geometric $n$-braids given by the braid diagrams in which the $i$-th and $(i+1)$-st strands cross once positively. In the following, we will always identify braid words with the corresponding geometric braids, and we suppress $n$ if the context is clear.

By gluing the top ends of the (oriented) strands of a geometric braid $\gamma \in B_n$ to the corresponding bottom ends, we get a knot (or link) $\widehat{\gamma}$, called the \emph{closure} of $\gamma$. 
If $\gamma$ induces a permutation with only one cycle on the ends of its $n$ strands, then its closure $\widehat{\gamma}$ is a knot and we call it an \emph{$n$-braid knot}. Note that conjugate braids $\gamma_0, \gamma_1 \in B_n$, denoted by $\gamma_0 \sim \gamma_1$, have isotopic closures $\widehat{\gamma_0}=\widehat{\gamma_1}$.
For a more detailed account on braids, we refer the reader to \cite{birmanbrendle}.

A \textit{positive braid} is an element of the braid group $B_n$ for some $n$ that can be written as a positive braid word $\sigma_{s_1} \sigma_{s_2} \cdots \sigma_{s_l}$ with $s_i \in \{1, \dots, n-1\}$.
A knot is called a \textit{positive braid knot} if it can be represented as the closure of a positive braid. 
The set of positive braid knots contains the sets of (positive) torus knots and algebraic knots, while itself being a subset of the set of positive knots or, more generally, the frequently studied set of (strongly) quasipositive knots.\\

Let $\operatorname{wr}(\gamma)$ denote the \emph{writhe} of a braid word $\gamma \in B_n$, \ie the exponent sum of the word $\gamma$. 
If $\gamma$ is a positive $n$-braid such that $K =\widehat{\gamma}$ is a knot, then, by work of Bennequin \cite{bennequin} and Rudolph \cite{rudolphslice} --- the latter building on Kronheimer and Mrowka's proof of the local Thom conjecture \cite{kronheimermrowka} --- we have
\begin{align}\label{eq:sliceBennequin}
g_4(K)=g(K)=\frac{\operatorname{wr}(\gamma)-n+1}{2}.
\end{align}

\subsection{The concordance invariants \texorpdfstring{$\tau$}{tau} and \texorpdfstring{$\Upsilon$}{Upsilon}}\label{sec:upsilonprops} 

In \cite{Ozsv_th_2003}, Ozsv\'ath and Szab\'o constructed the knot invariant $\tau$ via the knot filtration on the Heegaard Floer chain complex of $S^3$; the latter was also defined independently by Rasmussen \cite{rasmussen2003floer}.
 The invariant $\tau$ induces a group homomorphism $\CC \to \Z$ from the (smooth) knot concordance group $\CC $ to the group of integers $\Z$ and gives a lower bound on the $4$-ball genus $g_4(K)$: we have $\left|\tau(K)\right|\leq g_4(K)$ for any knot $K$. For the torus knots $T_{p,q}$, where $p$ and $q$ are coprime positive integers, the invariant $\tau$ recovers the $3$-genus \cite[Corollary 1.7]{Ozsv_th_2003}, namely we have 
\begin{align}\label{eq:tauTorus}
\tau \left(T_{p,q}\right) =g\left(T_{p,q}\right)=\frac{(p-1)(q-1)}{2}.
\end{align}
Moreover, it follows from \cite[Theorem 4 and Corollary 7]{livingston04} together with \Cref{eq:sliceBennequin} above that, for any knot $K$ that is the closure of a positive $n$-braid $\gamma$, we have
\begin{align}\label{eq:tauGamma}
\tau(K)=\frac{\operatorname{wr}(\gamma)-n+1}{2}=g_4(K)=g(K).
\end{align}

The invariant $\Upsilon$ was defined by  Ozsv\'ath, Stipsicz and Szab\'o in \cite{OSS2017}. 
We will not recall the definition of $\Upsilon$ via the knot Floer complex $CFK^\infty(K)$ since the properties of $\Upsilon$ mentioned below will be enough for our later computations and we will not explicitly use the Heegaard Floer theory behind it. 
For an overview on the properties of $\Upsilon$, see the original article \cite{OSS2017} or Livingston's notes on $\Upsilon$ \cite{livingstonnotes}; see \cite{homsurvey} for a survey on Heegaard Floer homology and knot concordance.

For every knot $K$, the knot invariant $\Upsilon_K \colon [0,1] \to \R$ 
is a continuous, piecewise linear function with the following properties \cite{OSS2017}:
\begin{align}
&\Upsilon_K(0)=0,\\
&\text{the slope of }\Upsilon_K(t) \text{ at } t = 0 \text{ is given by } -\tau(K),\\
&\Upsilon_{K_1\#K_2}(t) = \Upsilon_{K_1}(t) + \Upsilon_{K_2}(t) \text{ for all } 0 \leq t \leq 1 \text{ and all knots } K_1 \text{ and } K_2, \label{additivity}\\
&\Upsilon_{-K}(t)=-\Upsilon_K(t) \text{ for all } 0 \leq t \leq 1,\label{mirror}\\
&\left|\Upsilon_K(t)\right|\leq g_4(K) t \text{ for all } 0\leq t \leq 1.\label{eq:Upsilon4genus2}
\end{align}
Here, $-K$ is the knot obtained by mirroring $K$ and reversing its orientation. Its concordance class is the inverse of the class of $K$ in the knot concordance group $\CC$.
It follows from \eqref{additivity}-\eqref{eq:Upsilon4genus2} that $\Upsilon$ induces a homomorphism from the concordance group to the group of real-valued piecewise linear functions on the interval $[0,1]$.\\ 

For some classes of knots, the invariant $\Upsilon$ can be explicitly computed in terms of classical knot invariants like the signature and the Alexander polynomial.

\begin{prop}[{\cite[Theorem 1.14]{OSS2017}}]\label{prop:alternating}
We have $\Upsilon_K(t) = \frac{\sigma(K)}{2} t$ for all alternating or quasi-alternating knots $K$ and all $0 \leq t \leq 1$.
\end{prop}

For positive torus knots, $\Upsilon_K(t)$ is completely determined by a combinatorial formula in terms of their Alexander polynomial \cite[Theorem 1.15]{OSS2017}. For torus knots of braid index $2$ or $3$, the following holds; see \eg \cite{Feller_2016}.
For $\ell \geq 0$, we have 
\begin{align}\label{lem:Upsilontorus2}
\Upsilon_{T_{2,2\ell+1}}(t)=-\tau\left(T_{2,2\ell+1}\right) \cdot t 
=-\ell\cdot t \qquad \text{for all }0\leq t \leq 1.
\end{align}
For $\ell \geq 0$ and $ k \in \{1,2\}$, we have
\begin{align}\label{lem:Upsilontorus}
\Upsilon_{T_{3,3\ell+1}}(1)&=\Upsilon_{T_{3,3\ell+2}}(1)+1=-2\ell,&\\
\Upsilon_{T_{3,3\ell+k}}(t)&=-\tau(T_{3,3\ell+k}) t=-(3\ell+k-1)t & \text{for all } 0\leq t \leq \frac{2}{3} \qquad \text{and}\nonumber\\
\Upsilon_{T_{3,3\ell+k}}(t) &\text{ is linear on } \left[\frac{2}{3},1\right].& \nonumber
\end{align}

\section{The Garside normal form for $3$-braids}\label{sec:normalform} 

In this section, we provide a classification result on the conjugacy classes of $3$-braids; see \Cref{prop:normalform}.
This result is basically due to work of Garside \cite{GARSIDE} 
who gave the first solution to the conjugacy problem for all braid groups $B_n$, $n \geq 3$, in 1965.
\Cref{prop:normalform} might be known to the experts, but since the explicit formulas appear to be missing from the literature,
we will provide them here.
\\

Throughout, we denote the two generators of the braid group $B_3$ by $a\eqdef \sigma_1$ and $b\eqdef \sigma_2$ which are 
subject to the braid relation $aba=bab$. Recall that the braid $\Delta^2=(aba)^2=(ab)^3$ generates the center of $B_3$.

\begin{rem}\label{rem:avsb}
Any $3$-braid is conjugate to the same braid with generators $a$ and $b$ interchanged. More precisely, let $\gamma = a^{p_1}b^{q_1}\cdots a^{p_r}b^{q_r}$ for some $r \geq 1$ and integers $p_i,q_i$, $i \in \{1, \dots, r\},$ be a $3$-braid. Then using $\Delta a = b \Delta$ and $\Delta b = a \Delta$, we have
\begin{align*}
\gamma &= \Delta^{-1} \Delta a^{p_1}b^{q_1}\cdots a^{p_r}b^{q_r} = \Delta^{-1}  b^{p_1}a^{q_1}\cdots b^{p_r}a^{q_r} \Delta \sim  
 b^{p_1}a^{q_1}\cdots b^{p_r}a^{q_r}.
\end{align*}
\end{rem}

In \Cref{prop:normalform}, we will provide a certain standard form for the conjugacy classes of $3$-braids.

\begin{prop}\label{prop:normalform}
Let $\gamma$ be a $3$-braid.
Then $\gamma$ is conjugate to one
of the $3$-braids
\begin{align}
&\Delta^{2\ell}a^p & \text{for }  \ell \in \Z,\, p \geq 0,\tag{A}\label{eq:linkcase} \\
&\Delta^{2\ell}a^pb &\text{for }  \ell \in \Z,\, p \in \{1,2,3\},\tag{B}\label{eq:torusknotcase}\\
&\Delta^{2\ell}a^{p_1}b^{q_1}\cdots a^{p_r}b^{q_r}  & \text{for } \ell \in \Z, \,r\geq 1,\, p_i,q_i\geq 2, \,i \in \{ 1, \dots, r\},\tag{C}\label{eq:evenpower}\\
&\Delta^{2\ell+1}a^{p_1}b^{q_1}\cdots a^{p_{r-1}}b^{q_{r-1}}a^{p_r} & \text{for }  \ell \in \Z,\,r\geq 1, \,p_r \geq 2,\, p_i, q_i\geq 2,\tag{D}\label{eq:oddpower} 
\\&&
i \in \{ 1, \dots, r-1\}.\nonumber
\end{align}
If $\gamma$ is a positive $3$-braid, then $\ell \geq 0$. 
If $\widehat{\gamma}$ is a knot, then only the cases \eqref{eq:torusknotcase}--\eqref{eq:oddpower} can occur and $p$ must be odd in case \eqref{eq:torusknotcase}, at least one of the $p_i$ and one of the $q_i$ must be odd in case \eqref{eq:evenpower}, and at least one of the $p_i$ or $q_i$ must be odd in case \eqref{eq:oddpower}.
\end{prop}

While we will never use it in this article, we note --- without proof --- the following uniqueness result related to \Cref{prop:normalform}.

\begin{rem}
Up to cyclic permutation of the powers $p_1,  q_1, \dots, p_r, q_r$ in \eqref{eq:evenpower} and $p_1,q_1, \dots, p_{r-1}, q_{r-1}, p_r$ 
in \eqref{eq:oddpower}, respectively, each $3$-braid is conjugate to exactly one of the $3$-braids listed in \Cref{prop:normalform}. 
This follows from Garside's work \cite{GARSIDE}.  
In his notation, each of the $3$-braids listed in \eqref{eq:linkcase}--\eqref{eq:oddpower} in \Cref{prop:normalform} is the standard form of a certain element in the (so-called) summit set of $\gamma$. For $3$-braids of the form \eqref{eq:evenpower} or \eqref{eq:oddpower}, the summit set consists of those $3$-braids obtained by cyclic permutation of the powers $p_1,  q_1, \dots, p_r, q_r$ in \eqref{eq:evenpower} and $p_1,q_1, \dots, p_{r-1}, q_{r-1}, p_r$ in \eqref{eq:oddpower}, respectively. 
\end{rem}

\begin{defn}\label{def:garsidenormalform}
We call a braid word of the form in \eqref{eq:linkcase}--\eqref{eq:oddpower} a \emph{$3$-braid in Garside normal form}.
\end{defn}

\begin{rem}
The advantage of the Garside normal form over Murasugi's normal form for $3$-braids used later in \Cref{sec:upsilongeneral} (see \Cref{def:murasuginormalform}) is that positive $3$-braids are easier to detect in this normal form: if $\gamma$ is a positive $3$-braid, then $\gamma$ is conjugate to one of the braids in \eqref{eq:linkcase}--\eqref{eq:oddpower} with $\ell \geq 0$. 
Since Garside's solution to the conjugacy problem works for any $n$-braid with $n \geq 3$,
one might hope to generalize an explicit standard form as in \Cref{prop:normalform} to $n$-braids for any $n \geq 3$. 
\end{rem}

\begin{rem}\label{rem:toruscasenormalform}
For odd $p$, case \eqref{eq:torusknotcase} of \Cref{prop:normalform} covers the torus knots of braid index $3$. More precisely, if $\gamma \sim \Delta^{2\ell}ab = (ab)^{3\ell+1}$, then its closure is $\widehat{\gamma}= T_{3,3\ell+1}$ for $\ell \geq 0$ and $\widehat{\gamma}= -T_{3,3(-\ell-1)+2}$ for $\ell < 0$, and if $\gamma \sim \Delta^{2\ell}a^3 b \sim (ab)^{3\ell+2}$, then $\widehat{\gamma}= T_{3,3\ell+2}$ for $\ell \geq 0$ and $\widehat{\gamma}= -T_{3,3(-\ell-1)+1}$ for $\ell < 0$.
\end{rem}

\begin{proof}[Proof of \Cref{prop:normalform}]
The proof will follow from the following claim.
\begin{claim} \label{claim:posbraids}
Let $\gamma$ be a positive $3$-braid.
Then $\gamma$ is conjugate to one of the $3$-braids in \eqref{eq:linkcase}--\eqref{eq:oddpower} with $\ell \geq 0$.
\end{claim}
We first deduce \Cref{prop:normalform} from this claim.
To that end, let $\gamma$ be any $3$-braid. 
If $\gamma$ is a positive braid, we are done by \Cref{claim:posbraids}. If not, then $\gamma$ can be written in the form  $\gamma = \Delta^m \alpha$ where $m$ is a negative integer and $\alpha$ a positive $3$-braid \cite[Theorem 5]{GARSIDE}.
In fact, inserting $\Delta^{-1}\Delta$ if $m$ is odd, we can assume $\gamma$ to be of the form $\Delta^{-2n} \alpha$ for some $n \geq 1$ and a positive $3$-braid $\alpha$. The proposition then easily follows using the claim for $\alpha$. It remains to prove \Cref{claim:posbraids}.

\begin{claimproof}{Proof of \Cref{claim:posbraids}}
A positive $3$-braid $\gamma$ has the form $\gamma = a^{P_1}b^{Q_1}\cdots a^{P_R}b^{Q_R}$ for integers $R \geq 1$, $P_i,Q_i \geq 0$, $i \in \{1, \dots, R\}$. 
If all the $P_i$ or all the $Q_i$ are $0$,
then (possibly using \Cref{rem:avsb}) $\gamma$ is conjugate to $a^p$ for some $p \geq 0$ and we are in case \eqref{eq:linkcase} for $\ell = 0$. 
Possibly after conjugation and reduction of $R$, we can thus assume that all of the integers $P_i, Q_i$ are non-zero. 
If $P_1, Q_1 \geq 2$ applies for all $i \in \{1, \dots, R\}$, then $\gamma$ is of the form in \eqref{eq:evenpower} for $\ell = 0$. If $R=1$, \ie  $\gamma = a^{P_1}b^{Q_1}$ for integers $P_1, Q_1 \geq 1$, and $P_1 = 1$ or $Q_1 = 1$, then (possibly using \Cref{rem:avsb}) $\gamma$ is conjugate to a braid of the form in \eqref{eq:torusknotcase}.

It remains to consider the case where $R \geq 2$ and at least one of the $P_i$ or $Q_i$ is $1$. 
In that case --- if necessary after conjugation --- $\gamma$ contains $\Delta = aba = bab$ as a subword
and is thus conjugate to $\Delta \alpha$ for some positive $3$-braid $\alpha$. Now, let $n \geq 1$ be maximal 
with the property that $\gamma$ is conjugate to
$\Delta^{n} \alpha$ for some positive $3$-braid $\alpha$. Then, possibly after conjugation of $\gamma$, the braid word $\alpha$ must be one of the following:
\begin{align}\label{eq:gammaprcases}
&a^{p} &\text{for } p\geq 0,\nonumber\\
&a^{p}b & \text{for } p\geq 1,\nonumber\\
&a^{p_1}b^{q_1}\cdots a^{p_r}b^{q_r}& \text{for } r\geq 1,\,p_i,q_i \geq 2,\,i \in \{1, \dots, r\},\\
&a^{p_1}b^{q_1}\cdots a^{p_{r-1}}b^{q_{r-1}}a^{p_r} &\text{for } r\geq 1,\,p_r\geq 2,\, p_i,q_i \geq 2,\,
i \in \{1, \dots, r-1\}.\nonumber
\end{align}
Indeed, using \Cref{rem:avsb}, up to conjugation these are the only possible words such that $\Delta^{n} \alpha$ does not contain any additional $\Delta$ as a subword. Note that $\alpha$ can be the empty word, which is covered by the first case in \eqref{eq:gammaprcases} for $p = 0$. Further, note that 
\begin{align}\label{eq:posbraidcasesobs}
\begin{split}
&\Delta^{2\ell }a^pb \sim \Delta^{2\ell +1} a^{p-2}, \qquad \Delta^{2\ell +1} a \sim \Delta^{2\ell }a^3b, \qquad \Delta^{2\ell +1} a^{p}b \sim \Delta^{2\ell +1} a^{p+1},\\
&\Delta^{2\ell +1} a^{p_1}b^{q_1}\cdots a^{p_r}b^{q_r} \sim \Delta^{2\ell +1} a^{p_1+q_r} b^{q_1}a^{p_2} \cdots b^{q_{r-1}}a^{p_r} \qquad \text{and}\\
&\Delta^{2\ell } a^{p_1}b^{q_1}\cdots a^{p_{r-1}}b^{q_{r-1}}a^{p_r} \sim \Delta^{2\ell } a^{p_1+p_r}b^{q_1}a^{p_2} \cdots a^{p_{r-1}}b^{p_{r-1}}
\end{split}
\end{align}
for any $\ell  \geq 0$, $p \geq 1$, $p_i, q_i \geq 2$, $i \in \{1, \dots, r\}$. 
It follows from a case by case analysis of the cases in \eqref{eq:gammaprcases}, using \eqref{eq:posbraidcasesobs} and taking the parity of $n$ into account, that any positive $3$-braid is conjugate to one of the $3$-braids in \eqref{eq:linkcase}--\eqref{eq:oddpower} with $\ell \geq 0$. 
\end{claimproof}
This concludes the proof of \Cref{prop:normalform}.
\end{proof}

\section{The upsilon invariant of $3$-braid knots}\label{sec:Upsilon}

In this section, we prove \Cref{thm:upsilon}.
Along the way, we compute the invariant $\upsilon$ for positive $3$-braid knots in Garside normal form (\Cref{prop:posupsilon}) and prove \Cref{cor:switches} and \Cref{cor:conc}.

\subsection{Methodology}\label{sec:methodology}

We first recall inequality \eqref{eq:upsilonboundintro} from the introduction --- which will be repeatedly used in \Cref{sec:Upsilon} --- in more generality.\\

The \emph{cobordism distance} $d(K,J)$ between two knots $K$ and $J$ is defined as the $4$-genus $g_4(K \# -J)$ of the connected sum of $K$ and the inverse of $J$. 
Equivalently, the cobordism distance $d(K,J)$ could be defined as the minimal genus of a smoothly and properly embedded oriented surface $C$ in $S^3 \times [0,1]$ with boundary $K \times \{0\} \cup J \times \{1\}$ such that the induced orientation on the boundary of $C$ agrees with the orientation of $K$ and disagrees with the orientation of $J$.
Suppose the genus of a cobordism $C$ between two knots $K$ and $J$ is $g(C)$. We then have $d(K,J) \leq g(C)$, so by the properties
\eqref{additivity}-\eqref{eq:Upsilon4genus2} of $\Upsilon$ from \Cref{sec:upsilonprops}
we get 
\begin{align}\label{eq:strategy}
\left\vert\Upsilon_K(t)-\Upsilon_{J}(t)\right\vert = 
\left\vert\Upsilon_{K \# -J}(t)\right\vert
\leq g_4(K \# -T)  t =d(K,T)t \leq g(C)  t
\end{align}
for all $0 \leq t \leq 1$.
This provides bounds on $\Upsilon_K(t)$ in terms of $\Upsilon_J(t)$ and $g(C)$.\\ 

We now give an example for the cobordisms we will use later on.

\begin{ex}\label{ex:introstrategy}
Among other things, we will frequently use the following trick the author first saw in \cite[Example 4.5]{Feller_2017}. 
Let $\gamma$ be a $3$-braid such that $K = \widehat{\gamma}$ is a knot. Consider the $3$-braid $\alpha\eqdef \gamma  b^{2n}$ for some $n \geq 1$. Then $\widehat{\alpha}$ is also a knot and there is a cobordism between $\widehat{\alpha}$ and the connected sum $ K\#T_{2,2n+1} $ of genus $1$. This cobordism can be realized by two saddle moves ($1$-handle attachments) of the form shown in \Cref{fig:saddlemoveintro}, performed in the two circled regions of \Cref{fig:cobintro}. 
 One of them is used to add a generator $b$ to the braid $\alpha$ to obtain the braid word $\gamma b^{2n+1} $ and the other is used to  transform the closure of this new braid word into a connected sum of $K$ and $T_{2,2n+1}$. Recall that our braid diagrams are oriented from bottom to top.

\captionsetup[subfigure]{labelfont=normalfont, labelformat = simple}
\begin{figure}[htbp]
     \centering
     \begin{subfigure}[t]{1\textwidth}
         \centering
         \begin{psfrags}
  \psfrag{a}{\small{$\gamma$}}
  \psfrag{b}{\small{$2n$}}
  \psfrag{c}{\small{$\widehat{\alpha}$}}
  \psfrag{d}{\begin{tabular}{@{}l@{}}
\small{$2n$}  \\ \small{$+1$}\end{tabular}}
  \psfrag{e}{\small{$\widehat{\gamma}\#T_{2,2n+1} $}}
\psfrag{f}{\begin{tabular}{@{}l@{}}
\small{$2$ saddle} \\\small{moves}\end{tabular}}
  \includegraphics[width=0.6 \textwidth]{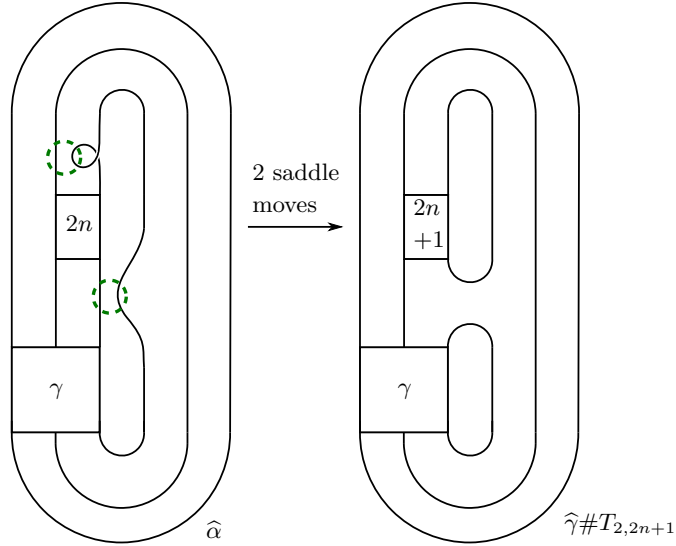}
\end{psfrags}
         \caption{A schematic of a cobordism between the knots $\widehat{\alpha}$ and $ \widehat{\gamma}\#T_{2,2n+1} $
        realized by two saddle moves.}\label{fig:cobintro}
     \end{subfigure}
          \begin{subfigure}[t]{1\textwidth}
         \centering
         \includegraphics[width=0.25 \textwidth]{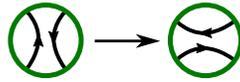}
\caption{A saddle move.}\label{fig:saddlemoveintro}
     \end{subfigure}
        \caption{An example illustrating our proof strategy.}\label{fig:exintro}
\end{figure}

Using $\upsilon\left(T_{2,2n+1}\right)=-n$ by \Cref{lem:Upsilontorus2} and that the genus of the cobordism is $1$, by \eqref{eq:strategy} for $t=1$, we have
\begin{align}\label{eq:exintro}
\left\vert  \upsilon\left(\widehat{\alpha}\right) - \upsilon\left( K\#T_{2,2n+1}\right)\right\vert \leq 1 \quad \Longleftrightarrow \quad 
\left\vert \upsilon\left(\widehat{\alpha}\right) - \upsilon\left(K\right) +n\right\vert \leq 1,
\end{align}
which provides the lower bound $\upsilon(K) \geq 
\upsilon\big(\widehat{\alpha}\big) +n -1$ on $\upsilon(K)$.
\end{ex}

\subsection{The upsilon invariant of positive $3$-braid knots}\label{sec:upsilonpos}

In this section, we determine the invariant $\upsilon$ for all positive $3$-braid knots.
\\ 

By \Cref{prop:normalform} and \Cref{rem:toruscasenormalform}, positive $3$-braid knots are either the torus knots $T_{3,3\ell+k}$ for $\ell \geq 0$ and $k \in \{1,2\}$ which have braid representatives of Garside normal form \eqref{eq:torusknotcase}, or closures of positive
$3$-braids of Garside normal form \eqref{eq:evenpower} or \eqref{eq:oddpower} (cf.~\Cref{def:garsidenormalform}).
The following proposition thus proves \Cref{thm:upsilon} for all positive $3$-braid knots.

\begin{prop}\label{prop:posupsilon}
Let $\gamma$ be a positive $3$-braid such that $K=\widehat{\gamma}$ is a knot. Then
\begin{align*}
\upsilon(K)  = \begin{cases}
-2\ell-\dfrac{p-1}{2}& \text{if } \gamma \text{ is conjugate to a braid in } \eqref{eq:torusknotcase}, \\
-\dfrac{\sum\limits_{i=1}^r \left(p_i + q_i\right)}{2} +r -2\ell  & \text{if } \gamma \text{ is conjugate to a braid in } \eqref{eq:evenpower}, \\
-\dfrac{\sum\limits_{i=1}^{r-1} \left(p_i + q_i\right)+p_r}{2} +r -2\ell-\dfrac{3}{2}& \text{if } \gamma \text{ is conjugate to a braid in } \eqref{eq:oddpower}.
\end{cases}
\end{align*}
\end{prop}

\begin{rem}
In fact, the formulas from \Cref{prop:posupsilon} also give the correct upsilon invariant in terms of the Garside normal form of a $3$-braid representative of a knot $K$ if $K$ is the closure of any $3$-braid in Garside normal form \eqref{eq:evenpower} or \eqref{eq:oddpower}, not necessarily a positive one. This follows from \Cref{thm:upsilon} (proved in the next section) and the observations of \Cref{sec:comparisonnormalforms}.
\end{rem}

Recall that for the torus knots of braid index $3$, we know the invariant $\upsilon$ by \Cref{lem:Upsilontorus}. In the following, we will determine the invariant $\upsilon$ for all knots that are closures of positive $3$-braids of Garside normal form \eqref{eq:evenpower} or \eqref{eq:oddpower}.

We first provide an upper bound on $\Upsilon_K(t)$ for positive $3$-braid knots $K$ and $0 \leq t \leq 1$. The following inequality \eqref{eq:upsilonUpperBound} in \Cref{lem:upsilonUpperBound} could also be shown using the dealternating number and a result of Abe and Kishimoto \cite[Lemma 2.2]{abekishimoto}, whereas the main work for the upper bound on $\upsilon$ for the knots in the second and third case in \Cref{prop:posupsilon} will be to rewrite the braid words representing these knots.
 We use the approach below since it will also give bounds on the minimal cobordism distance between any positive $3$-braid knot and an alternating knot; see \Cref{rem:Ag}.

\begin{lem}\label{lem:upsilonUpperBound}
Let $\gamma = a^{p_1}b^{q_1}\cdots a^{p_r}b^{q_r}$ be a positive $3$-braid, where $r\geq 1$ and $p_i,q_i \geq 1$, $i \in \{1, \dots, r\}$, are integers such that $K=\widehat{\gamma}$ is a knot.
Then 
\begin{align}\label{eq:upsilonUpperBound}
\Upsilon_K(t)  \leq \left(-g(K)+r-1\right) t
\qquad \text{for all }0 \leq t \leq 1.
\end{align}
\end{lem}

\begin{proof}
We claim that there is a cobordism $C$ of genus
\begin{align}\label{eq:genusgeps}
 g(C) = \frac{r-1+\varepsilon}{2}
\end{align}
between $K$ and the connected sum
\begin{align*}
J_{\varepsilon} = T_{2,\sum\limits_{i=1}^r p_i +\varepsilon_p} \# \,T_{2,q_1+\varepsilon_1} \#\, T_{2,q_2+\varepsilon_2} \# \,\dots \#\, T_{2,q_r+\varepsilon_r},
\end{align*}
where $\varepsilon_1,\dots,\varepsilon_r, \varepsilon_p \in \{0,1\}$ are chosen such that $J_{\varepsilon}$ is a connected sum of torus knots (rather than links), \ie such that $\sum_{i=1}^r p_i +\varepsilon_p$, $q_1+\varepsilon_1$, $q_2+\varepsilon_2,\dots,q_r+\varepsilon_r$ are all odd, and $\varepsilon \coloneqq \varepsilon_p + \sum_{i=1}^r \varepsilon_i$. This cobordism $C$ can be realized by $r-1+\varepsilon$ saddle moves as follows. 
Following the schematic in \Cref{fig:schematicCob}, we add $\varepsilon$ generators $b$ by $\varepsilon$ saddle moves and additionally perform $r-1$ saddle moves of the form shown in \Cref{fig:exintro}(b) in the green circled regions of \Cref{fig:schematicCob}. 
In \Cref{fig:schematicCob}, a box on the left labeled $p_i$ or $q_i$ stands for the positive braid $a^{p_i}$ or $b^{q_i}$, respectively.
The Euler characteristic of the cobordism $C$ is $\chi(C) = -r+1-\varepsilon$. Since $C$ is connected and --- as $J_\varepsilon$ and $K$ are knots --- has two boundary components, the genus of $C$ is $g(C) = \frac{-\chi(C)}{2} = \frac{r-1+\varepsilon}{2}$ as claimed. 
\begin{figure}[htbp]
  \centering
  \begin{psfrags}
  \psfrag{a}{\small{$p_1$}}
  \psfrag{b}{\small{$q_1$}}
  \psfrag{c}{\small{$p_2$}}
  \psfrag{d}{\small{$q_2$}}
  \psfrag{e}{\small{$p_r$}}
  \psfrag{f}{\small{$q_r$}}
\psfrag{g}{\begin{tabular}{@{}l@{}}
\small{  $q_1$}\\
\small{$+\varepsilon_1$}\end{tabular}}
  \psfrag{h}{\begin{tabular}{@{}l@{}}
\small{  $q_2$}\\
\small{$+\varepsilon_2$}\end{tabular}}
  \psfrag{j}{\begin{tabular}{@{}l@{}}
\small{  $q_r$}\\
\small{$+\varepsilon_r$}\end{tabular}}
\psfrag{k}{\begin{tabular}{@{}l@{}}
$r-1+\varepsilon$ \\saddle \\moves\end{tabular}}
\psfrag{m}{\begin{tabular}{@{}l@{}}
\small{  $p_r$}\\
\small{$+\varepsilon_p$}\end{tabular}}
\psfrag{n}{\small{$K$}}
\psfrag{p}{\small{$J_\varepsilon$}}
  \includegraphics[width=0.63\textwidth]{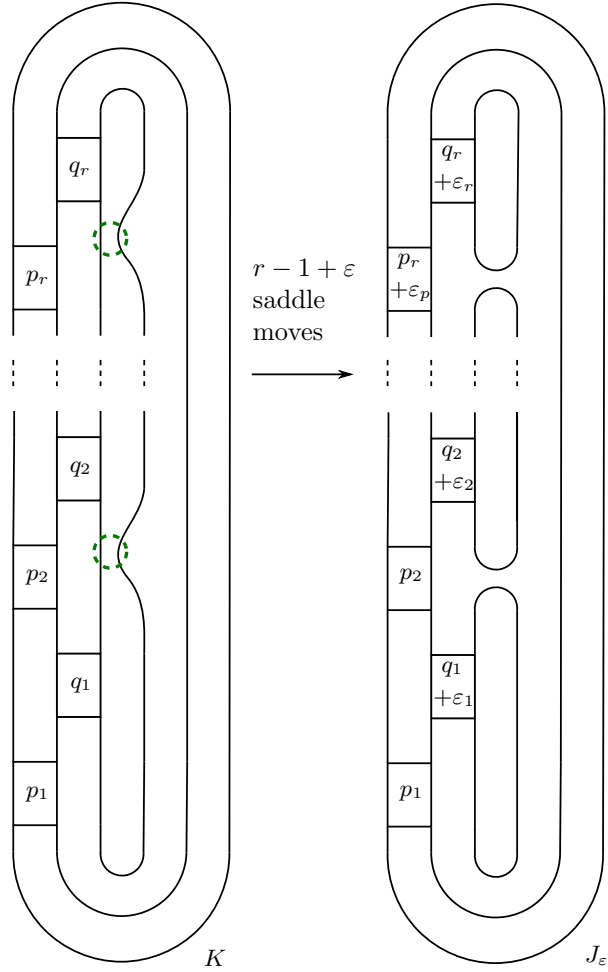}
\end{psfrags}
\caption{A schematic of a cobordism between the knots $K = \widehat{\gamma}$ and $J_{\varepsilon}= T_{2,\sum_{i=1}^r p_i +\varepsilon_p} \# T_{2,q_1+\varepsilon_1} \# T_{2,q_2+\varepsilon_2} \# \dots \# T_{2,q_r+\varepsilon_r}$ realized by $r-1+\varepsilon$ saddle moves.}  
\label{fig:schematicCob}
\end{figure}

By \eqref{eq:strategy}, we get $\left\vert\Upsilon_K(t)-\Upsilon_{J_{\varepsilon}}(t)\right\vert \leq 
g(C)  t$ for all $0 \leq t \leq 1$, hence
\begin{align}\label{eq:upslowerbound}
\Upsilon_K(t) &\leq  \Upsilon_{J_{\varepsilon}}(t)+g (C) t\qquad \text{for all }0\leq t \leq 1.
\end{align}
By \Cref{additivity,lem:Upsilontorus2} from \Cref{sec:upsilonprops}, we have 
\begin{align*}
\Upsilon_{J_{\varepsilon}}(t)&
= \left(-\frac{\sum_{i=1}^r p_i +\varepsilon_p-1}{2}-\frac{q_1+\varepsilon_1-1}{2}-\frac{q_2+\varepsilon_2-1}{2}\dots -\frac{q_r+\varepsilon_r-1}{2}\right)t\\
&
= -\frac{1}{2}\left(\sum\limits_{i=1}^r (p_i + q_i)  -(r+1)+\varepsilon\right)t,
\end{align*}
so \eqref{eq:genusgeps} and \eqref{eq:upslowerbound} imply
\begin{align*}
\Upsilon_K(t)  \leq \left(-\frac{\sum\limits_{i=1}^r (p_i + q_i) }{2}+r\right) t 
\qquad
\text{for all }0 \leq t \leq 1.
\end{align*}
The claim follows, since by \Cref{eq:sliceBennequin}, 
we have
\begin{align*}
g(K) &=\frac{\operatorname{wr}(\gamma)-2}{2}=\frac{\sum\limits_{i=1}^r (p_i + q_i) -2}{2}.\qedhere
\end{align*}
\end{proof}

The following two lemmas improve the bound from \Cref{lem:upsilonUpperBound} for knots that are closures of positive $3$-braids of Garside normal form \eqref{eq:evenpower} or \eqref{eq:oddpower}, respectively.

\begin{lem}\label{lem:Upsilon1_Delta2l+1Upper}
Let $\gamma = \Delta^{2\ell+1}a^{p_1}b^{q_1}\cdots a^{p_{r-1}}b^{q_{r-1}}a^{p_r}$ for some $\ell \geq 0$, $ r\geq 1,$ $p_r \geq 1$ and $p_i, q_i\geq 1$ for $i\in \{1, \dots, r-1\}$ such that $K = \widehat{\gamma}$ is a knot. Then
\begin{align*}
\Upsilon_K(t) 
 \leq \left(-\frac{\sum\limits_{i=1}^{r-1}(p_i + q_i)+p_r }{2}+r-2\ell-\frac{3}{2} \right) t
\qquad
\text{for all } 0 \leq t \leq 1.
\end{align*}
\end{lem}

In the proof of \Cref{lem:Upsilon1_Delta2l+1Upper}, we will use that in $B_3$, we have 
\begin{align}\label{lem:braideq1}
(ab)^{3n +1} = ab \Delta^{2n} =
a^2ba^3(aba^3)^{n -1}ba^{n } \qquad \text{for all } n  \geq  1, 
\end{align}
where $\Delta^2 = (aba)^2 = (ab)^3=(ba)^3$; see \cite[Proof of Prop. 22]{Feller_2016}.

\begin{proof}[Proof of \Cref{lem:Upsilon1_Delta2l+1Upper}]
Let $\Sigma_\gamma =\sum\limits_{i=1}^{r-1}(p_i + q_i)+p_r$ and note that using \Cref{eq:sliceBennequin}, we have
\begin{align}\label{eq:genusK}
g(K) = \frac{3(2\ell+1)+ \Sigma_\gamma -2}{2} = \frac{\Sigma_\gamma }{2}+3\ell+\frac{1}{2}.
\end{align}
If $\ell = 0$, then $\gamma = \Delta a^{p_1}b^{q_1}\cdots a^{p_{r-1}}b^{q_{r-1}}a^{p_r}$ is conjugate to 
\begin{align*}
\gamma_1= 
a^{p_1+1}b^{q_1}\cdots a^{p_{r-1}}b^{q_{r-1}}a^{p_r+1}b
\end{align*}
and $\widehat{\gamma_1}=\widehat{\gamma}=K$, so $g\left(\widehat{\gamma_1}\right)=\frac{\Sigma_\gamma }{2}+\frac{1}{2}$. 
By \Cref{lem:upsilonUpperBound}, we get
\begin{align*}
\Upsilon_K(t) \leq \left(- g\left(\widehat{\gamma_1}\right)+r -1\right) t =
 \left(-\frac{\Sigma_\gamma }{2}+r -\frac{3}{2}\right)t \qquad
\text{for all } 0 \leq t \leq 1.
\end{align*}
For $\ell \geq 1$, using $\Delta^{2\ell + 1}  = (ab)^{3\ell} aba = (ab)^{3\ell +1} a$, we have 
\begin{align*}
\gamma &= \Delta^{2\ell+1}a^{p_1}b^{q_1}\cdots a^{p_{r-1}}b^{q_{r-1}}a^{p_r} =(ab)^{3\ell +1} a^{p_1+1}b^{q_1}\cdots a^{p_{r-1}}b^{q_{r-1}}a^{p_r}\\
&\overset{\eqref{lem:braideq1}}{=}a^2ba^3(aba^3)^{\ell-1}ba^{p_1+\ell+1}b^{q_1}\cdots a^{p_{r-1}}b^{q_{r-1}}a^{p_r}\\
&\sim a^{p_r+2}ba^3(aba^3)^{\ell-1}ba^{p_1+\ell+1}b^{q_1}\cdots a^{p_{r-1}}b^{q_{r-1}}=:\gamma_1.
\end{align*}
We have $\widehat{\gamma_1}=\widehat{\gamma}=K$ and
$
g\left(\widehat{\gamma_1}\right) =  \frac{\Sigma_\gamma }{2}+3\ell+\frac{1}{2}
$
by \eqref{eq:genusK}. Again, \Cref{lem:upsilonUpperBound} implies
\begin{align*}
\Upsilon_K(t) \leq 
\left( - g\left(\widehat{\gamma_1}\right)+r+\ell -1\right) t =
 \left(-\frac{\Sigma_\gamma }{2}+r - 2\ell-\frac{3}{2}\right) t \qquad \text{for all } 0 \leq t \leq 1,
\end{align*}
which proves the claim of the lemma.
\end{proof}

\begin{lem}\label{lem:Upsilon1_Delta2l_upperBound}
Let $\gamma = \Delta^{2\ell}a^{p_1}b^{q_1}\cdots a^{p_r}b^{q_r}$ for some $\ell \geq 0$, $r\geq 1$ and $p_i,q_i\geq 1$ for $i\in \{1, \dots, r\}$ such that $K = \widehat{\gamma}$ is a knot.
Then
\begin{align*}
\Upsilon_K(t) 
 \leq \left(-\frac{\sum\limits_{i=1}^r (p_i + q_i) }{2}+r-2\ell\right) t
\qquad
\text{for all } 0 \leq t \leq 1.
\end{align*}
\end{lem}

In the proof, we will need the following statement about positive $3$-braids.

\begin{lem}\label{lem:braideq2}
In $B_3$, we have 
\begin{align}\label{eq:braideq2}
(ab)^{3n -1} = a^{2n }b(a^2b^2)^{n -1}a  \qquad \text{for all } n  \geq 1.
\end{align}
\end{lem} 

\begin{proof}
Starting with the left-hand side we have
\begin{align*}
(ab)^{3n -1} = a(ba)^{3(n -1)}bab 
=a(ab)^{3(n -1)}aba,
\end{align*}
which proves \Cref{eq:braideq2} for $n  = 1$. We now show by induction that
\begin{align}\label{eq:braidind}
(ab)^{3(n -1)}a = a^{2n -1}b(a^2b^2)^{n -2}a^2b \qquad \text{for all } n  \geq 2,
\end{align}
which implies the lemma for all $n \geq 1$. For $n  = 2$, we have
\begin{align*}
(ab)^3 a  = a (ba)^3 = a(ab)^3 = a^2babab = a^3ba^2b.
\end{align*}
Assuming that \eqref{eq:braidind} is true for some $n -1 \geq 2$, we get
\begin{align*}
(ab)^{3(n -1)}a & =a(ba)^{3(n -1)} = a(ab)^{3(n -1)} = 
a^2(ba)^{3(n -2)}babab= a^2(ab)^{3(n -2)}aba^2b 
\\
&= a^2\left(a^{2n -3}b(a^2b^2)^{n -3}a^2b\right)ba^2b= a^{2n -1}b(a^2b^2)^{n -2}a^2b,
\end{align*}
using the induction hypothesis in the second to last equality. 
\end{proof}

\begin{proof}[Proof of \Cref{lem:Upsilon1_Delta2l_upperBound}]
Let $\Sigma_\gamma =\sum\limits_{i=1}^r (p_i + q_i)$. 
If $\ell = 0$, then by \Cref{eq:sliceBennequin} and \Cref{lem:upsilonUpperBound} we have 
\begin{align*}
\Upsilon_K(t) \leq  \left(- g\left(K\right)+r -1\right)t = \left(-\frac{\Sigma_\gamma }{2}+r\right)t \qquad
\text{for all } 0 \leq t \leq 1.
\end{align*}
For $\ell \geq 1$, using $\Delta^{2}= (ba)^{3}$ and \Cref{lem:braideq2}, we have 
\begin{align*}
\gamma &=  (ba)^{3\ell}a^{p_1}b^{q_1}\cdots a^{p_r}b^{q_r}\sim (ab)^{3\ell-1}a^{p_1+1}b^{q_1}\cdots a^{p_r}b^{q_r+1}\\
 &\sim 
a^{2\ell}b(a^2b^2)^{\ell-1}a^{p_1+2}b^{q_1}\cdots a^{p_r}b^{q_r+1}=:\gamma_1.
\end{align*}
Note that $\widehat{\gamma_1}=\widehat{\gamma}=K$ and by \Cref{eq:sliceBennequin}, we have 
\begin{align*}
g\left(\widehat{\gamma_1}\right) = g(K) = \frac{6\ell  + \Sigma_\gamma -2}{2} = \frac{\Sigma_\gamma }{2}+3\ell-1.
\end{align*}
Again by \Cref{lem:upsilonUpperBound}, we get
\begin{align*}
\Upsilon_K(t) \leq 
\left(- g\left(\widehat{\gamma_1}\right)+r+\ell -1\right) t =
\left(  -\frac{\Sigma_\gamma }{2}+r - 2\ell\right) t \qquad \text{for all } 0 \leq t \leq 1.&\qedhere
\end{align*}
\end{proof}

We will now focus on $\upsilon(K) = \Upsilon_K(1)$ and prove \Cref{prop:posupsilon} by showing that the upper bounds on $\Upsilon_K(t)$ from \Cref{lem:Upsilon1_Delta2l+1Upper} and \Cref{lem:Upsilon1_Delta2l_upperBound} for $t=1$ are also lower bounds. 
We will need the following observation used in \cite[Example 4.5]{Feller_2017} about $3$-braids, which we prove here for completeness.
\begin{lem}\label{lem:braideq3}
In $B_3$, we have
\begin{align}\label{eq:braids}
a^{2n+1}b\left(a^2b^2\right)^{n}=(ab)^{3n+1} \, \text{and } \,
b^{2n+1}a\left(b^2a^2\right)^{n}=(ba)^{3n+1}
\qquad \text{for all } n\geq 0.
\end{align}
\end{lem}

\begin{proof}
We prove the first statement by induction.
For $n=0$, the equality is clearly true. For $n=1$, using $ \Delta a= b \Delta $ and $\Delta b=a \Delta$, we have 
\begin{align*}
a^3b a^2b^2 = a^2 \Delta a b^2 = a^2 ba \Delta b = a \Delta^2 b = \Delta^2 ab = (ab)^{4}.
\end{align*}
We now assume that \eqref{eq:braids} is true for some $n-1 \geq 0$. Using the induction hypothesis and the equality for $n=1$, we get
\begin{align*}
a^{2n+1}b\left(a^2b^2\right)^{n}
 &= a^2 \left(ab\right)^{3(n-1)+1} a^2 b^2=a^3 b \Delta^{2(n-1)}a^2 b^2\\&=\Delta^{2(n-1)} a^3b a^2b^2 =(ab)^{3(n-1)}(ab)^{4}=(ab)^{3n+1}.\qedhere
\end{align*}
\end{proof}

\begin{lem}\label{lem:Upsilon1_Delta2l+1}
Let $\gamma = \Delta^{2\ell+1}a^{p_1}b^{q_1}\cdots a^{p_{r-1}}b^{q_{r-1}}a^{p_r}$ for some $\ell \geq 0$, $ r\geq 1,$ $p_r \geq 3$ and $p_i, q_i\geq 2$ for $i\in \{1, \dots, r-1\}$ such that $K=\widehat{\gamma}$ is a knot.
Then
\begin{align*}
\upsilon(K)&=-\frac{\sum\limits_{i=1}^{r-1}(p_i + q_i)+p_r}{2}+r-2\ell-\frac{3}{2}.
\end{align*}
\end{lem}

\begin{proof}[Proof of \Cref{lem:Upsilon1_Delta2l+1}]
Let $\Sigma_\gamma =\sum\limits_{i=1}^{r-1}(p_i + q_i)+p_r$.
From \Cref{lem:Upsilon1_Delta2l+1Upper}, it follows directly that
$
\upsilon(K) = \Upsilon_K(1) \leq -\frac{\Sigma_\gamma }{2}+r-2\ell-\frac{3}{2},
$
so we are left to show that $\upsilon(K) \geq -\frac{\Sigma_\gamma }{2}+r-2\ell-\frac{3}{2}$. To that end, consider
\begin{align*}
\gamma &= \Delta^{2\ell+1}a^{p_1}b^{q_1}\cdots a^{p_{r-1}}b^{q_{r-1}}a^{p_r} 
\sim \Delta^{2\ell}a\Delta a^{p_1}b^{q_1}\cdots a^{p_{r-1}}b^{q_{r-1}}a^{p_r-1}\\
&= \Delta^{2\ell}bab^2 a^{p_1}b^{q_1}\cdots a^{p_{r-1}}b^{q_{r-1}}a^{p_r-1}
=:\gamma_1,
\end{align*}
where we used $a\Delta = abab= bab^2$.
Note that $\widehat{\gamma_1} = \widehat{\gamma}=K$.
Now, define
\begin{align*}
\alpha := b^{2r} \gamma_1= \Delta^{2\ell}b^{2r+1} ab^2 a^{p_1}b^{q_1}\cdots a^{p_{r-1}}b^{q_{r-1}}a^{p_r-1}
\end{align*}
and note that $\widehat{\alpha}$ is a knot.
By assumption, we have $p_r -1 \geq 2$.
There is a cobordism between $\widehat{\alpha}$ and the connected sum $ T_{2,2r+1}\#\widehat{\gamma_1} = T_{2,2r+1}\# K$ of genus $1$ by using two saddle moves similar to the two saddle moves illustrated in \Cref{fig:exintro} from \Cref{ex:introstrategy}.
Similarly as in \eqref{eq:exintro} from \Cref{ex:introstrategy}, we have
$\upsilon(K) \geq 
\upsilon\big(\widehat{\alpha}\big) +r -1$.
In order to find a lower bound for $\upsilon\big(\widehat{\alpha}\big)$, 
note that there is a cobordism $C$ between $\widehat{\alpha}$ and the torus knot $T=T_{3,3(\ell + r)+1}$ of genus $g(C) =\frac{\Sigma_\gamma }{2}-2r+\frac{1}{2}$. Here we think of $T$ as the closure of the braid word $\beta = \Delta^{2\ell}b^{2r+1}a(b^2a^2)^{r}$, which 
 is equal to $\Delta^{2\ell} (ba)^{3r+1}=(ba)^{3(\ell +r) +1}$ as $3$-braids by \Cref{lem:braideq3}. 
The cobordism $C$ between $\widehat{\alpha}$ and $T=\widehat{\beta}$ can thus be realized by
\begin{align*}
p_1-2+q_1-2+\dots+p_{r-1}-2+q_{r-1}-2+p_r-3 
= \Sigma_\gamma -4r+1
\end{align*}
saddle moves corresponding to the deletion of the same number of generators $a$ and $b$ from the braid word $\alpha$ to obtain $\beta$. 
Hence the Euler characteristic of the cobordism $C$ is $\chi(C) = -\Sigma_\gamma +4r-1$. Since $C$ is connected and has two boundary components (as $\widehat{\alpha}$ and $T=\widehat{\beta}$ are knots), the genus of $C$ is indeed $g(C) =\frac{\Sigma_\gamma }{2}-2r+\frac{1}{2}$.
Now, by \eqref{eq:strategy}
and 
\Cref{lem:Upsilontorus}, we have
\begin{align*}
\upsilon\left(\widehat{\alpha}\right) \geq \upsilon\left( T \right) - g(C) = -2\left(\ell + r \right) -\left( \frac{\Sigma_\gamma }{2}-2r+\frac{1}{2}\right) = -\frac{\Sigma_\gamma }{2}-2\ell-\frac{1}{2}.
\end{align*}
It follows that
\begin{align*}
\upsilon(K) \geq 
\upsilon\left(\widehat{\alpha}\right) +r -1 \geq 
-\frac{\Sigma_\gamma }{2}+r-2\ell-\frac{3}{2}
\end{align*}
as claimed, hence the statement of the lemma.
\end{proof}

\begin{lem}\label{lem:Upsilon1_Delta2lCase1}
Let $\gamma = \Delta^{2\ell}a^{p_1}b^{q_1}\cdots a^{p_r}b^{q_r}$ for some $\ell \geq 0$, $r\geq 1$, $p_r, q_r\geq 3$ and $p_i,q_i\geq 2$ for $i\in \{1, \dots, r-1\}$
such that $K=\widehat{\gamma}$ is a knot. Then
\begin{align*}
\upsilon(K)&=-\frac{\sum\limits_{i=1}^r (p_i + q_i) }{2}+r-2\ell.
\end{align*}
\end{lem}

\begin{proof}[Proof of \Cref{lem:Upsilon1_Delta2lCase1}]
The proof uses similar ideas as the proof of \Cref{lem:Upsilon1_Delta2l+1}.
Let $\Sigma_\gamma =\sum\limits_{i=1}^r (p_i + q_i)$. 
By \Cref{lem:Upsilon1_Delta2l_upperBound}, we have $\upsilon(K) \leq -\frac{\Sigma_\gamma }{2}+r-2\ell$, so it remains to show that $\upsilon(K) \geq -\frac{\Sigma_\gamma }{2}+r-2\ell$. To that end, we consider
\begin{align*}
\gamma =\Delta^{2\ell}a^{p_1}b^{q_1}\cdots a^{p_r}b^{q_r} \sim  \Delta^{2\ell}ba^{p_1}b^{q_1}\cdots a^{p_r}b^{q_r-1}=:\gamma_1.
\end{align*}
Note that $\widehat{\gamma_1}=\widehat{\gamma}=K$. 
We define
\begin{align*}
\alpha := a^{2r} \gamma_1= a^{2r}  \Delta^{2\ell}ba^{p_1}b^{q_1}\cdots a^{p_r}b^{q_r-1} \sim 
\Delta^{2\ell} b a^{2r}b a^{p_1}b^{q_1}\cdots a^{p_r}b^{q_r-2}=:\alpha_1.
\end{align*}
Then $\widehat{\alpha_1}= \widehat{\alpha}$ is a knot and by assumption we have $q_r -2 \geq 1$. There is a cobordism between $\widehat{\alpha}$ and $ T_{2,2r+1}\#\widehat{\gamma_1} = T_{2,2r+1}\# K$ of genus $1$ by using two saddle moves similar to the cobordism considered in \Cref{ex:introstrategy} and in the proof of \Cref{lem:Upsilon1_Delta2l+1}, hence 
$\upsilon(K) \geq \upsilon\big(\widehat{\alpha_1}\big) +r -1$.
To find a lower bound for $\upsilon\big(\widehat{\alpha_1}\big)$, we observe that there is a cobordism $C$ between the knot $\widehat{\alpha_1}$ and the knot $\widehat{\beta}$, where 
\begin{align*}
\beta  = \Delta^{2\ell} b a^{2r}b (a^2b^2)^{r-1} a^3 b.
\end{align*}
Using \Cref{eq:braids} from \Cref{lem:braideq3} for $n-1$, in $B_3$, we have
\begin{align*}
b a^{2n}b \left(a^2b^2\right)^{n-1} a^2 &= 
ba(ab)^{3(n-1)+1} a^2 = ba \Delta^{2(n-1)}aba^2
= \Delta^{2n}
\qquad \text{for all } n \geq 1.
\end{align*}
We thus have
$
\beta = \Delta^{2\ell} \Delta^{2r} a b = (ab)^{3(\ell+r)+1}
$, so the closure of $\beta$ is the torus knot $T=T_{3,3(\ell + r)+1}$ with $\upsilon(T) = -2\left(\ell + r \right)$ by \Cref{lem:Upsilontorus}. 
The cobordism $C$ between $\widehat{\alpha_1}$ and $T=\widehat{\beta}$ can be realized by
\begin{align*}
p_1-2+q_1-2+\dots+p_{r-1}-2+q_{r-1}-2+p_r-3+q_r-3
= \Sigma_\gamma -4r-2
\end{align*}
saddle moves corresponding to the deletion of the same number of generators $a$ and $b$ from the braid word $\alpha_1$ to obtain $\beta$. By a similar Euler
characteristic argument as in the proofs of \Cref{lem:upsilonUpperBound} and \Cref{lem:Upsilon1_Delta2l+1}, the genus of this cobordism is $g(C) =\frac{\Sigma_\gamma }{2}-2r-1$. Note that here we used $p_r \geq 3$ and $q_r \geq 3$.
Now, by \eqref{eq:strategy}, we have 
\begin{align*}
\upsilon(\widehat{\alpha_1}) &\geq \upsilon\left( T \right) - g(C) 
= -\frac{\Sigma_\gamma }{2}-2\ell+1, \qquad \text{hence}\\
\upsilon(K) &\geq 
\upsilon\left(\widehat{\alpha_1}\right) +r -1 \geq 
-\frac{\Sigma_\gamma }{2}+r-2\ell.\qedhere
\end{align*}
\end{proof}

\begin{lem}\label{lem:Upsilon1_Delta2lCase2}
Let $\gamma = \Delta^{2\ell}a^{p_1}b^{q_1}\cdots a^{p_r}b^{q_r}$ for some $\ell \geq 0$, $r\geq 2$, $p_i,q_i\geq 2$ for $i\in \{1, \dots, r\}$. Suppose that $q_r\geq 3$ and $p_k \geq 3$ for some $1\leq k < r$ 
and that $K=\widehat{\gamma}$ is a knot. Then
\begin{align*}
\upsilon(K)&=-\frac{\sum\limits_{i=1}^r (p_i + q_i) }{2}+r-2\ell.
\end{align*}
\end{lem}

\begin{proof}
We proceed similar as in the proof of \Cref{lem:Upsilon1_Delta2lCase1}, but here we will look at a different cobordism to obtain a lower bound for $\upsilon\left(\widehat{\alpha_1}\right)$.
The steps of the proof are exactly the same until then, so we consider
\begin{align*}
\gamma =\Delta^{2\ell}a^{p_1}b^{q_1}\cdots a^{p_r}b^{q_r} \sim  \Delta^{2\ell}ba^{p_1}b^{q_1}\cdots a^{p_r}b^{q_r-1}=:\gamma_1
\end{align*}
and define 
\begin{align*}
\alpha := a^{2r} \gamma_1
\sim 
\Delta^{2\ell} b a^{2r}b a^{p_1}b^{q_1}\cdots a^{p_r}b^{q_r-2}=:\alpha_1.
\end{align*}
Again, we have
$\upsilon(K) \geq \upsilon(\widehat{\alpha_1}) +r -1.
$ 
Now, in order to find a lower bound for $\upsilon(\widehat{\alpha_1})$,
we observe that there is a cobordism $C$ between $\widehat{\alpha_1}$ and the knot $\widehat{\beta}$, where 
\begin{align*}
\beta  = \Delta^{2\ell} b a^{2r}b (a^2b^2)^{k-1} a^3 b^2(a^2b^2)^{r-k-1}a^2b.
\end{align*}
We find the cobordism $C$ by the deletion of generators from the braid word $\beta$ to obtain $\alpha_1$, where we use the assumptions $q_r \geq 3$ and $p_k \geq 3$.
In fact, the cobordism can be realized by
\begin{align*}
&p_1-2+q_1-2+\dots+p_{k-1}-2+q_{k-1}-2+p_k-3+q_k-2\\ &\quad +p_{k+1}-2+q_{k+1}-2+\dots+p_{r-1}-2+q_{r-1}-2+p_r-2+q_r-3 \\
&=
 \Sigma_\gamma -4r-2
\end{align*}
saddle moves, so its genus is $g(C)=\frac{\Sigma_\gamma }{2}-2r-1$.
Using $a^{2k-1}b (a^2b^2)^{k-1} =(ab)^{3k-2}$ 
by \Cref{lem:braideq3}, we have 
\begin{align*}
\beta &= \Delta^{2\ell} b a^{2r-2k+1} (ab)^{3k-2} a^3 b^2(a^2b^2)^{r-k-1}a^2b\\&= \Delta^{2\ell} b a^{2r-2k+1} \Delta^{2(k-1)}ab a^3 b^2(a^2b^2)^{r-k-1}a^2b\\
&\sim\Delta^{2(\ell+k-1)} \Delta a^2 b^2(a^2b^2)^{r-k-1}a^2b^2 a^{2r-2k+1} \\
&= \Delta^{2(\ell+k-1)+1}  (a^2b^2)^{r-k+1} a^{2r-2k+1}=:\beta_1.
\end{align*}
Note that by our assumptions on $\ell$, $r$ and $k$, we have $\ell+k-1 \geq 0$, $r-k+1 \geq 2$ and $2r-2k+1 \geq 3$, so $\beta_1$ has the form of the braid words considered in \Cref{lem:Upsilon1_Delta2l+1}. We thus have
\begin{align*}
\upsilon\left(\widehat{\beta}\right)&=\upsilon\left(\widehat{\beta_1}\right) = - \frac{4(r-k+1)+2r-2k+1}{2}+\left(r-k+2\right)-2\left(\ell+k-1\right) -\frac{3}{2}\\
&
=-2(\ell+r).
\end{align*}
By \eqref{eq:strategy}, we have 
\begin{align*}
\upsilon(\widehat{\alpha_1}) &\geq \upsilon\left( \widehat{\beta} \right) - g(C) 
= -\frac{\Sigma_\gamma }{2}-2\ell+1, \qquad \text{hence}\\
\upsilon(K) &\geq 
\upsilon\left(\widehat{\alpha_1}\right) +r -1 \geq 
-\frac{\Sigma_\gamma }{2}+r-2\ell.\qedhere
\end{align*}
\end{proof}

\begin{proof}[{Proof of \Cref{prop:posupsilon}}]
The first case of \Cref{prop:posupsilon} follows from \Cref{rem:toruscasenormalform} and \Cref{lem:Upsilontorus}.
\Cref{lem:Upsilon1_Delta2lCase1} and \ref{lem:Upsilon1_Delta2lCase2} together 
prove the second case,
\Cref{lem:Upsilon1_Delta2l+1} proves the third case. 
Note that up to conjugation, by \Cref{rem:avsb} and the remarks in \Cref{prop:normalform}, it is no restriction to assume that $p_r \geq 3$ in \Cref{lem:Upsilon1_Delta2l+1} and 
that $q_r \geq 3$ and either
$p_r \geq 3$ or $p_k \geq 3$ for some $1 \leq k < r$ in \Cref{lem:Upsilon1_Delta2lCase1} and \ref{lem:Upsilon1_Delta2lCase2}, respectively.
\end{proof}

Before we proceed with the general case where the knot $K$ is given as the closure of any $3$-braid, let us prove the following corollaries of our results in this section.

\begin{cor}[\Cref{cor:switches}]\label{cor:switches2}
Let $K$ be a knot that is the closure of a positive $3$-braid. Then
\begin{align*}
r = g(K) + \upsilon(K)+1
\end{align*}
is minimal among all integers $r \geq 1$ such that $K$ is the closure of a positive $3$-braid 
$ a^{p_1}b^{q_1}\cdots a^{p_r}b^{q_r}$ for integers $p_i,q_i \geq 1$, $i \in \{1, \dots, r\}$.
\end{cor}

\begin{proof}
By \Cref{lem:upsilonUpperBound} we have
\begin{align*}
\upsilon(K)
 \leq 
-g(K)+r-1 \quad \Longleftrightarrow \quad g(K) + \upsilon(K) +1 \leq r 
\end{align*} 
whenever $K$ is the closure of a positive $3$-braid $a^{p_1}b^{q_1}\cdots a^{p_r}b^{q_r}$ for integers $r\geq 1$, $p_i,q_i \geq 1$, $i \in \{1, \dots, r\}$.
It remains to show that we can always find a positive braid representative for $K$ of the form $a^{p_1}b^{q_1}\cdots a^{p_r}b^{q_r}$ with $r = g(K) + \upsilon(K) +1$. We will use \Cref{prop:normalform}.
In fact, if $K$ is the closure of a positive braid $\gamma$ of the form in \eqref{eq:evenpower} with $\ell \geq 0$, then $g \left(K\right) +\upsilon\left(K\right) +1 = r + \ell$ by 
\Cref{eq:sliceBennequin} applied to $\gamma$, \Cref{lem:Upsilon1_Delta2lCase1} and \Cref{lem:Upsilon1_Delta2lCase2}.
Moreover, we have
\begin{align*}
\gamma &= a^{p_1}b^{q_1}\cdots a^{p_{r}}b^{q_{r}} \qquad \text{if } \ell = 0 \qquad \text{and}\\
\gamma &\sim  a^{2\ell}b(a^2b^2)^{\ell-1}a^{p_1+2}b^{q_1}\cdots a^{p_r}b^{q_r+1} \qquad \text{if } \ell \geq 1
\end{align*}
by the proof of \Cref{lem:Upsilon1_Delta2l_upperBound};
these give the desired braid representatives for $K$.
Furthermore, if $K$ is represented by a positive braid $\gamma$ of the form in \eqref{eq:oddpower} with $\ell \geq 0$, then $g \left(K\right) +\upsilon\left(K\right) +1 = r + \ell$ by \Cref{eq:sliceBennequin} and  \Cref{lem:Upsilon1_Delta2l+1}, and we have
\begin{align*}
 \gamma &\sim a^{p_1+1}b^{q_1}\cdots a^{p_{r-1}}b^{q_{r-1}}a^{p_r+1}b \qquad \text{if } \ell = 0 \qquad \text{and}\\
\gamma & \sim a^{p_r+2}ba^3(aba^3)^{\ell-1}ba^{p_1+\ell+1}b^{q_1}\cdots a^{p_{r-1}}b^{q_{r-1}} \qquad \text{if } \ell \geq 1
\end{align*}
by the proof of \Cref{lem:Upsilon1_Delta2l+1Upper}.
Finally, if $K= T_{3,3\ell+k}$ for $\ell \geq 0$ and $k \in \{1,2\}$, then by \Cref{eq:tauTorus} and \Cref{lem:Upsilontorus}, we have
$ g(K) + \upsilon(K) +1
=\ell+1$
and $T_{3,3\ell+1}$ and $T_{3,3\ell+2}$ are represented by the positive $3$-braids
$
(ab)^{3\ell+1} =a^{2\ell+1}b\left(a^2b^2\right)^{\ell}$ and $
(ab)^{3\ell+2}\sim a^{2\ell+3 }b(a^2b^2)^{\ell } ,
$
respectively, by \Cref{lem:braideq3} and \Cref{lem:braideq2}.
\end{proof}

\begin{cor}[\Cref{cor:conc}]
If $K$ and $J$ are concordant knots that are both closures of positive $3$-braids, then
the minimal $r$ from \Cref{cor:switches2} is the same for both $K$ and $J$.
\end{cor}

\begin{proof}
If $K$ and $J$ are concordant, then 
their $4$-genus and their upsilon invariants are equal. So by \Cref{eq:sliceBennequin} from \Cref{sec:braids} and by \Cref{cor:switches2}, 
positive $3$-braids with closures $K$ and $J$, respectively, 
 will have the same minimal $r$. 
\end{proof}

\begin{rem}\label{rem:Ag}
Let $\mathcal{A}_g(K)$ denote the minimal genus of a cobordism between a knot $K$ and an alternating knot, \ie the cobordism distance $d\left(K, \{\text{alternating knots}\}\right)$. By \cite[Theorem 8]{friedllivingstonzentner}, we have 
$
\frac{\left\vert \tau(K) + \upsilon(K) \right \vert}{2} \leq
\mathcal{A}_g(K) 
$
for any knot $K$.
It thus follows from our results in this section that
\begin{align*}
\frac{r+\ell-1}{2} \leq \mathcal{A}_g(K) \leq \frac{r+\ell-1+\varepsilon}{2}
\end{align*}
for any knot $K$ that is the closure of a positive $3$-braid in Garside normal form \eqref{eq:evenpower} or \eqref{eq:oddpower}, where $\varepsilon \geq 0$ is an integer depending on $K$. The lower bound uses
\Cref{prop:posupsilon} and \Cref{eq:tauGamma} from \Cref{sec:upsilonprops}; 
see also the proof of \Cref{cor:switches2}. The upper bound follows from
the proofs of \Cref{lem:Upsilon1_Delta2l+1Upper} and \Cref{lem:Upsilon1_Delta2l_upperBound}, see also the proof of \Cref{lem:upsilonUpperBound}.
Note that for most positive $3$-braid knots, we have $\varepsilon > 0$, so we do not get an equality.

A shorter proof of 
\Cref{lem:upsilonUpperBound} without cobordisms follows from a result of Abe and Kishimoto on the dealternating number of positive $3$-braid knots. 
 Indeed, we have
\begin{align*}
\left \vert \Upsilon_K(t)  + g(K)t \right \vert &\overset{\eqref{eq:tauGamma}}{=} \left \vert \Upsilon_K(t)  + \tau(K)t \right \vert 
\leq \alt(K) t\overset{\eqref{eq:altdalt}}{\leq} \dalt(K)t \\&\overset{\eqref{upperbounddalt}}{\leq} 
\left(r-1\right) t
\qquad
\text{for all }0 \leq t \leq 1.
\end{align*}
The definitions of the dealternating number $\dalt(K)$ and the alternation number $\alt(K)$ of a knot $K$ and more details on the inequalities used here will be provided in \Cref{sec:alternation}.\end{rem}

\subsection{Proof of \texorpdfstring{\Cref{thm:upsilon}}{Theorem 1.1}}\label{sec:upsilongeneral}

It remains to show \Cref{thm:upsilon} when $K$ is the closure of a not necessarily positive $3$-braid. We first recall a result of Murasugi, which implies that indeed all $3$-braid knots except for the torus knots of braid index $3$ are covered by \Cref{thm:upsilon}. 

Let $\gamma$ be a $3$-braid. Then, by \cite[Proposition 2.1]{murasugibook}, $\gamma$ is conjugate to one and only one of the $3$-braids
\begin{align}
&\Delta^{2\ell}a^p \quad \text{or}\quad  \Delta^{2\ell+1} & \text{for }  \ell \in \Z, \,p \in \Z,\tag{a}\label{eq:case1} \\
&\Delta^{2\ell}ab \quad \text{or}\quad \Delta^{2\ell}(ab)^2  & \text{for }  \ell \in \Z,\tag{b}\label{eq:case2} \\
&\Delta^{2\ell}a^{-p_1}b^{q_1}\cdots a^{-p_r}b^{q_r} & \text{for } \ell \in \Z, \,r\geq 1,\, p_i,q_i\geq 1, \,i \in \{ 1, \dots, r\}.\tag{c}\label{eq:case3} 
\end{align}

\begin{defn}\label{def:murasuginormalform}
We call a braid word of the form in \eqref{eq:case1}--\eqref{eq:case3} 
 a \emph{$3$-braid in Murasugi normal form}.
\end{defn}

\begin{rem}
The closures of the $3$-braids in Murasugi normal form \eqref{eq:case1} are links of two (if $p$ is odd) or three components and the closures of the $3$-braids in Murasugi normal form \eqref{eq:case2} are the torus knots of braid index $3$ (cf.~\Cref{rem:toruscasenormalform}).
\end{rem}

If $\ell = 0$ in case \eqref{eq:case3}, 
the braid word $\gamma = a^{-p_1}b^{q_1}\cdots a^{-p_r}b^{q_r}$ for integers $r\geq 1$ and $p_i,q_i\geq 1$, $i \in \{ 1, \dots, r\}$, gives rise to an alternating braid diagram. If $K = \widehat{\gamma}$ is a knot, by \Cref{prop:alternating} we thus have $\upsilon(K) = \frac{\sigma(K)}{2}$ in that case and the statement of \Cref{thm:upsilon} follows directly from a result by Erle on the signature of $3$-braid knots.

\begin{prop}[{\cite[Theorem 2.6]{erle}}]\label{prop:signature}
Let $\gamma = \Delta^{2\ell}a^{-p_1}b^{q_1}\cdots a^{-p_r}b^{q_r}$ for integers $\ell\in \Z$, $r\geq 1$ and $p_i,q_i\geq 1$ for $i\in \{1, \dots, r\}$
such that $K=\widehat{\gamma}$ is a knot. Then 
\begin{align*}
\sigma(K)&=\sum\limits_{i=1}^r (p_i - q_i) -4\ell.
\end{align*}
\end{prop}

We still need to show \Cref{thm:upsilon} when $K$ is the closure of a $3$-braid in Murasugi normal form \eqref{eq:case3} 
with $\ell \neq 0$. The proof will follow from the following two lemmas.

\begin{lem}\label{lem:Murasugiupper}
Let $\gamma = \Delta^{2\ell}a^{-p_1}b^{q_1}\cdots a^{-p_r}b^{q_r}$ for some $\ell\geq 1,r \geq 1$ and $p_i,q_i\geq 1$ for $i\in \{1, \dots, r\}$ such that $K=\widehat{\gamma}$ is a knot. Then
\begin{align*}
\Upsilon_K(t) &\leq  \left(\frac{\sum\limits_{i=1}^r (p_i - q_i) }{2}-2\ell\right) t
\qquad \text{for all }0\leq t \leq 1.
\end{align*}
\end{lem}

\begin{lem}\label{lem:Murasugi}
Let $\gamma = \Delta^{2\ell}a^{-p_1}b^{q_1}\cdots a^{-p_r}b^{q_r}$ for some $\ell \geq 0$, $r\geq 1$ and $p_i,q_i\geq 1$ for $i\in \{1, \dots, r\}$
such that $K=\widehat{\gamma}$ is a knot. Then
\begin{align*}
\upsilon(K)&\geq \frac{\sum\limits_{i=1}^r (p_i - q_i) }{2}-2\ell.
\end{align*}
\end{lem}

\begin{proof}[Proof of \Cref{thm:upsilon}]
For $\ell \geq 1$, the statement of the theorem follows directly from \Cref{lem:Murasugiupper} and \Cref{lem:Murasugi}.
If $\ell < 0$, the knot $-K$ is represented by the braid word
$
\Delta^{-2\ell} a^{-q_r}b^{p_r} \cdots a^{-q_1}b^{p_1}
$
with $-\ell \geq 1$ 
and accordingly we have
\begin{align*}
\upsilon\left(-K\right) = \frac{\sum\limits_{i=1}^r (q_i - p_i) }{2}+2\ell.
\end{align*}
Using that $\upsilon(-K) = - \upsilon(K)$ by \Cref{mirror} from \Cref{sec:upsilonprops}, this implies the claim.
\end{proof}

The remainder of this section is devoted to prove the above lemmas.

\begin{proof}[Proof of \Cref{lem:Murasugiupper}]
We first consider the case where $p_1 \geq 2$ and $\ell \geq 2$. 
Using $\Delta a^{-1} =ab$ and 
\begin{align*}
(ab)^{3n +2} =
b^{n+1} a(b^3ab)^{n-1} b^3 a b^3
\qquad \text{for all } n  \geq  1 \quad \text{\cite[Proof of Prop. 22]{Feller_2016}},
\end{align*}
 we have 
\begin{align*}
\gamma &=\Delta^{2\ell}a^{-p_1}b^{q_1}\cdots a^{-p_r}b^{q_r} 
= \Delta^{2(\ell-1)+1}ab a^{-p_1+1}b^{q_1}\cdots a^{-p_r}b^{q_r}\\
&= (ba)^{3(\ell-1)+2} ba^{-p_1+1}b^{q_1}\cdots a^{-p_r}b^{q_r} \sim (ab)^{3(\ell-1)+2}a^{-p_1+1}b^{q_1}\cdots a^{-p_r}b^{q_r+1} 
\\&\sim a(b^3ab)^{\ell-2} b^3 a b^3a^{-p_1+1}b^{q_1}\cdots a^{-p_r}b^{q_r+\ell+1}=: \gamma_1.
\end{align*}
Now, we claim that there is a cobordism $C$ of genus
$
g(C) = \frac{\ell+r-1 + \varepsilon}{2}
$
between the closure $K$ of $\gamma_1$ and the connected sum
\begin{align*}
J_{\varepsilon} =
-T_{2,p_1-1-\varepsilon_1} \,\#\, -T_{2,p_2-\varepsilon_2}\, \# \,\dots\, \#\, -T_{2,p_r-\varepsilon_r}
\# \,T_{2,\sum\limits_{i=1}^r q_i +5\ell-1+\varepsilon_q},
\end{align*}
where we choose $\varepsilon_1,\dots,\varepsilon_r, \varepsilon_q \in \{0,1\}$ such that $J_{\varepsilon}$ is a connected sum of torus knots, \ie such that $\sum_{i=1}^r q_i +5\ell-1+\varepsilon_q$, $p_1-1-\varepsilon_1$, $p_2-\varepsilon_2,\dots,p_r-\varepsilon_r$ are all odd; and $\varepsilon= \varepsilon_q + \sum_{i=1}^r \varepsilon_i$.
This cobordism $C$ can be realized using $\ell+r-1 + \varepsilon$ saddle moves as follows. On the one hand, we add $\sum_{i=1}^r \varepsilon_i$ generators $a$ and $\varepsilon_q$ generators $b$ to the braid word $\gamma_1$, on the other hand, we perform $\ell+r-1$ saddle moves of the form as the $r-1$ saddle moves used in the proof of \Cref{lem:upsilonUpperBound} to get a connected sum of torus knots.
The Euler characteristic of $C$ is $\chi(C) = -\ell-r+1-\varepsilon$. Since $C$ is connected and has two boundary components (as $K$ and $J_\varepsilon$ are knots), the genus of $C$ is $g(C) = \frac{-\chi(C)}{2} = \frac{\ell+r-1 + \varepsilon}{2}$ as claimed.
By \Cref{additivity,lem:Upsilontorus2}, we have
\begin{align*}
\Upsilon_{J_\varepsilon}(t) &
=\left(\frac{\sum\limits_{i=1}^r (p_i-q_i)-\varepsilon-r-5\ell+1}{2}\right)t \qquad \text{for all }0\leq t \leq 1
\end{align*}
and by \eqref{eq:strategy}, we get
\begin{align*}
\Upsilon_K(t) &\leq  \Upsilon_{J_{\varepsilon}}(t)+g (C) t = \left(\frac{\sum\limits_{i=1}^r (p_i - q_i) }{2}-2\ell\right) t
\qquad \text{for all }0\leq t \leq 1.
\end{align*}
If $p_1 \geq 2$ and $\ell = 1$, then 
\begin{align*}
\gamma \sim (ab)^{2}a^{-p_1+1}b^{q_1}\cdots a^{-p_r}b^{q_r+1} \sim ab^2 a^{-p_1+1}b^{q_1}\cdots a^{-p_r}b^{q_r+2} =: \gamma_1,
\end{align*}
and similarly as above, there is a cobordism $C$ of genus $g(C) = \frac{r + \varepsilon}{2}$
between the closure $K$ of $\gamma_1$ and the connected sum
\begin{align*}
J_{\varepsilon} =
-T_{2,p_1-1-\varepsilon_1} \,\#\, -T_{2,p_2-\varepsilon_2}\, \# \,\dots\, \#\, -T_{2,p_r-\varepsilon_r}
\# \,T_{2,\sum\limits_{i=1}^r q_i +4+\varepsilon_q},
\end{align*}
where we choose $\varepsilon_1,\dots,\varepsilon_r, \varepsilon_q \in \{0,1\}$ such that $J_{\varepsilon}$ is a connected sum of torus knots and $\varepsilon= \varepsilon_q + \sum_{i=1}^r \varepsilon_i$. The claim follows also in this case from Equations \eqref{additivity} and \eqref{lem:Upsilontorus2}, and the intequality in \eqref{eq:strategy}.

It remains to show the claim when $p_1=1$.
In that case, using $\Delta a^{-1} = ab$, we 
have
\begin{align*}
\gamma &=\Delta^{2\ell}a^{-1}b^{q_1}\cdots a^{-p_r}b^{q_r} = \Delta^{2\ell-1}ab^{q_1+1}\cdots a^{-p_r}b^{q_r} \sim \Delta^{2\ell-1}b^{q_1+1}\cdots a^{-p_r}b^{q_r+1}. 
\end{align*}
If $\ell =1$, then $\gamma$ is conjugate to
$
\gamma_1 = ab^{q_1+2}a^{-p_2}b^{q_2}\cdots a^{-p_r}b^{q_r+2}
$
and if $\ell \geq 2$, then using \Cref{lem:braideq1} from \Cref{sec:upsilonpos}, we have 
\begin{align*}
\gamma & \sim \Delta^{2(\ell-1)+1}b^{q_1+1}a^{-p_2}b^{q_2}\cdots a^{-p_r}b^{q_r+1}
= (ba)^{3(\ell-1)+1}b^{q_1+2}a^{-p_2}b^{q_2}\cdots a^{-p_r}b^{q_r+1}
\\&\sim  ab^3(bab^3)^{\ell-2}ab^{q_1+\ell+1}a^{-p_2}b^{q_2}\cdots a^{-p_r}b^{q_r+3}=:\gamma_1.
\end{align*}
In both cases, there is a cobordism $C$ of genus
$
g(C) =\frac{\ell+r-2 + \varepsilon}{2}
$
between the closure $K$ of $\gamma_1$ and the connected sum
\begin{align*}
J_{\varepsilon} = 
 -T_{2,p_2-\varepsilon_2}\, \# \,\dots\, \#\, -T_{2,p_r-\varepsilon_r}
\# \,T_{2,\sum\limits_{i=1}^r q_i +5\ell-1+\varepsilon_q},
\end{align*}
where we choose $\varepsilon_1,\dots,\varepsilon_r, \varepsilon_q \in \{0,1\}$ such that $J_{\varepsilon}$ is a connected sum of torus knots and $\varepsilon= \varepsilon_q + \sum_{i=1}^r \varepsilon_i$. 
Using \eqref{additivity}, \eqref{lem:Upsilontorus2}, and \eqref{eq:strategy} again, the claim follows.
\end{proof}

We will need the following two technical lemmas for the proof of \Cref{lem:Murasugi}.

\begin{lem}\label{lem:technical1}
Let $\gamma = \Delta^{2\ell}a^{p_1}b^{q_1}\cdots a^{p_r}b^{q_r}$ for some $\ell \geq 0$, $r\geq 1$ and integers $p_i,q_i$ such that $p_i < 0$ or $p_i \geq 2$, and $q_i < 0$ or $q_i \geq 2$, for any $i\in \{1, \dots, r\}$.
Moreover, assume that $K=\widehat{\gamma}$ is a knot. Then
\begin{align*}
\upsilon(K)&\geq -\frac{\sum\limits_{i=1}^r (p_i + q_i) }{2}+r-2\ell - \#\{ i \mid p_i < 0 \} - \#\{i \mid q_i < 0\},
\end{align*}
where $\# A$ denotes the cardinality of the set $A$.
\end{lem}

\begin{lem}\label{lem:technical2}
Let $\gamma = \Delta^{2\ell+1}a^{p_1}b^{q_1}\cdots a^{p_{r-1}}b^{q_{r-1}}a^{p_r}$ for some $\ell \geq 0$, $r\geq 1$ and integers $p_i,q_i$ such that $p_i < 0$ or $p_i \geq 2$ for any $i\in \{1, \dots, r\}$ and $q_i < 0$ or $q_i \geq 2$ for any $i\in \{1, \dots, r-1\}$.
Moreover, assume that $K=\widehat{\gamma}$ is a knot. Then
\begin{align*}
\upsilon(K)&\geq -\frac{\sum\limits_{i=1}^{r-1} (p_i + q_i)+p_r }{2}+r-2\ell -\frac{3}{2}- \#\{ i \mid p_i < 0 \} - \#\{i \mid q_i < 0\}.
\end{align*}
\end{lem}

For the proof of \Cref{lem:technical1} and \Cref{lem:technical2}, we refer the reader to the very end of this section; we will first prove \Cref{lem:Murasugi} using these lemmas. 

\begin{proof}[Proof of \Cref{lem:Murasugi}]
Let $k$ be the number of exponents $q_j$ of $\gamma$ with $q_j=1$ and let $\mathcal{J} = \{j_1, \dots, j_k\}$ for $0 \leq k \leq r$ be the set of indices such that $q_{j} =1$ if and only if $j \in \mathcal{J}$. For all $j \in \mathcal{J} $, we rewrite the subword $a^{-p_j} b^{q_j}$ of $\gamma$ using $\Delta^{-1}ab = a^{-1}$ as
\begin{align*}
a^{-p_j} b^{q_j} = a^{-p_j} b
= a^{-p_j} a^{-1} \Delta \Delta^{-1} a b = a^{-p_j -1} \Delta a^{-1} = \Delta b^{-p_j-1} a^{-1}.
\end{align*}
Note that if $j, j+1 \in \mathcal{J}$, then
$
a^{-p_j} b^{q_j}a^{-p_{j+1}} b^{q_{j+1}} = \Delta^2 a^{-p_j-1}b^{-p_{j+1}-2}a^{-1}.
$
After rewriting $a^{-p_j} b^{q_j}$ for all $j \in \mathcal{J} $, 
the braid $\gamma $ is conjugate to $ \gamma_1 = \Delta^{2\ell + k} \alpha$ for some $3$-braid $\alpha$ which is of the form
\begin{align*}
\alpha = \begin{cases}
a^{\widetilde{p_1}}b^{\widetilde{q_1}}
 \cdots a^{\widetilde{p_{n}}}b^{\widetilde{q_{n}}} \text{ for } n = r-\frac{k}{2}& \text{if } k \text{ is even},\\
b^{\widetilde{p_1}}a^{\widetilde{q_1}}
 \cdots b^{\widetilde{p_{n-1}}}a^{\widetilde{q_{n-1}}}b^{\widetilde{p_{n}}} \text{ for } n = r-\frac{k-1}{2}& \text{if } k \text{ is odd},
\end{cases} 
\end{align*}
where
$
\sum\limits_{i=1}^n (\widetilde{p_i} + \widetilde{q_i})  = \sum\limits_{i=1}^r \left( - p_i +q_i\right) -3k 
$ and where 
the $\widetilde{p_i}$ and $\widetilde{q_i}$ fulfill the assumptions of \Cref{lem:technical1} and \Cref{lem:technical2}, respectively, \ie where $\widetilde{p_i} < 0$ or $\geq 2$ and $\widetilde{q_i} < 0$ or $\geq 2$ for any $i$. 
The number of negative exponents in $\alpha$ equals the number of negative exponents $-p_i$ in $\gamma$, so 
\begin{align*}
\#\{ i\mid \widetilde{p_i} < 0 \} + \#\{ i \mid \widetilde{q_i} < 0 \}  =r.
\end{align*}
If $k$ is even, by \Cref{lem:technical1}, we get
\begin{align*}
  \upsilon\left( \widehat{\gamma}\right) &\geq 
  -\frac{\sum\limits_{i=1}^n (\widetilde{p_i} + \widetilde{q_i}) }{2}+n-(2\ell+k) - \#\{ i \mid \widetilde{p_i} < 0 \} + \#\{ i \mid \widetilde{q_i} < 0 \}\\
  &= -\frac{ \sum\limits_{i=1}^r \left( - p_i +q_i\right) -3k }{2}+r-\frac{k}{2}-(2\ell+k) - r
 = \frac{\sum\limits_{i=1}^r \left( p_i -q_i\right)  }{2}-2\ell.
\end{align*}
Similarly, if $k$ is odd, the claim follows from  \Cref{lem:technical2}.
\end{proof}

It remains to prove \Cref{lem:technical1} and \Cref{lem:technical2}.

\begin{proof}[Proof of \Cref{lem:technical1}]
We will modify the braid word $\gamma$ in $2r$ steps, where each step corresponds to one of  
the $2r$ exponents $p_i, q_i$, $i \in \{1, \dots, r\}$, of $\gamma$. In every step, we will either just conjugate $\gamma$ (if the corresponding exponent is positive) or perform a cobordism of genus $1$ between the closure of $a^{2n}\gamma$ or $b^{2n}\gamma$ and the connected sum $T_{2,2n+1}\# \widehat{\gamma}$ for some $n \geq 0$ --- similarly as the cobordism described in \Cref{ex:introstrategy} and used in the proofs of \Cref{lem:Upsilon1_Delta2l+1},  \Cref{lem:Upsilon1_Delta2lCase1} and \Cref{lem:Upsilon1_Delta2lCase2}. We now describe these steps in more detail.
First, let $\gamma_{0,q}\pr = \gamma$ and define
\begin{align*}
&a^{-p_1+2+\varepsilon_{1,p}} \gamma_{0,q}\pr = \Delta^{2\ell}a^{2+\varepsilon_{1,p}}b^{q_1}a^{p_2}b^{q_2}\cdots a^{p_r}b^{q_r} &
\\&\sim \Delta^{2\ell} b^{q_1}a^{p_2}b^{q_2}\cdots a^{p_r}b^{q_r}a^{2+\varepsilon_{1,p}} =: \gamma_{1,p}\pr & \text{if } p_1 < 0 \text{ and}\\
&\gamma_{0,q}\pr \sim \Delta^{2\ell} b^{q_1}a^{p_2}b^{q_2}\cdots a^{p_r}b^{q_r}a^{p_1} =:\gamma_{1,p}\pr &\text{if } p_1 > 0,
\end{align*}
so that 
$
\gamma_{1,p}\pr = \Delta^{2\ell} b^{q_1}a^{p_2}\cdots a^{p_r}b^{q_r}a^{\widetilde{p_1}}
$
for some $\widetilde{p_1} \geq 2$ (note that we assumed $p_1 < 0$ or $p_1 \geq 2$). Here, if $p_1 < 0$, we choose $\varepsilon_{1,p} \in \{0,1\}$ such that $-p_1 + 2 + \varepsilon_{1,p}$ is even and $\widehat{\gamma_{1,p}\pr}$ is a knot.
Second, let $\varepsilon_{1,q} \in \{0,1\}$ such that $-q_1 + 2 + \varepsilon_{1,q}$ is even if $q_1 < 0$, and define
\begin{align*}
\gamma_{1,q} &= 
b^{-q_1+2+\varepsilon_{1,q}} \gamma_{1,p}\pr
 = \Delta^{2\ell}b^{2+\varepsilon_{1,q}}a^{p_2}b^{q_2}\cdots a^{p_r}b^{q_r}a^{\widetilde{p_1} } &\\&\sim \Delta^{2\ell} a^{p_2}b^{q_2}\cdots a^{p_r}b^{q_r}a^{\widetilde{p_1} } b^{2+\varepsilon_{1,q}}=: \gamma_{1,q}\pr & \text{if } q_1 < 0 \text{ and}\\
\gamma_{1,q} &=
\gamma_{1,p}\pr \sim \Delta^{2\ell} a^{p_2}b^{q_2}\cdots a^{p_r}b^{q_r}a^{\widetilde{p_1} } b^{q_1} =:\gamma_{1,q}\pr & \text{if } q_1 > 0,
\end{align*}
so that 
$
\gamma_{1,q}\pr = \Delta^{2\ell} a^{p_2}b^{q_2}\cdots a^{p_r}b^{q_r}a^{\widetilde{p_1}}b^{\widetilde{q_1}}
$
for some $\widetilde{p_1}, \widetilde{q_1}\geq 2$.
Inductively, for any $1 \leq i \leq r$, we let
\begin{align*}
&a^{-p_i+2+\varepsilon_{i,p}} \gamma_{i-1,q}\pr 
= \Delta^{2\ell}a^{2+\varepsilon_{i,p}}b^{q_i}
a^{p_{i+1}}\cdots a^{p_r}b^{q_r}a^{\widetilde{p_1}}b^{\widetilde{q_1}}
 \cdots a^{\widetilde{p_{i-1}}}b^{\widetilde{q_{i-1}}} &
\\&\sim\Delta^{2\ell}b^{q_i}
a^{p_{i+1}}\cdots a^{p_r}b^{q_r}a^{\widetilde{p_1}}b^{\widetilde{q_1}}
 \cdots a^{\widetilde{p_{i-1}}}b^{\widetilde{q_{i-1}}} a^{2+\varepsilon_{i,p}} =: \gamma_{i,p}\pr &\text{if } p_i < 0 \text{ and}\\
&\gamma_{i-1,q}\pr \sim \Delta^{2\ell} b^{q_i}a^{p_{i+1}}\cdots a^{p_r}b^{q_r}a^{\widetilde{p_1}}b^{\widetilde{q_1}}
 \cdots a^{\widetilde{p_{i-1}}}b^{\widetilde{q_{i-1}}}a^{p_{i}} =:\gamma_{i,p}\pr & \text{if } p_i > 0,
\end{align*}
so that 
\begin{align*}
\gamma_{i,p}\pr = \Delta^{2\ell} b^{q_i}a^{p_{i+1}}\cdots a^{p_r}b^{q_r}a^{\widetilde{p_1}}b^{\widetilde{q_1}}
 \cdots a^{\widetilde{p_{i-1}}}b^{\widetilde{q_{i-1}}}a^{\widetilde{p_{i}}} 
\end{align*}
for some integers $\widetilde{p_{1}}, \widetilde{q_{1}} , \dots, \widetilde{p_{i-1}} , \widetilde{q_{i-1}} ,  \widetilde{p_{i}} \geq 2$. Here we choose $\varepsilon_{i,p} \in \{0,1\}$ such that $-p_i + 2 + \varepsilon_{i,p}$ is even if $p_i < 0$.
Moreover, for $1 \leq i \leq r$, we let  $\varepsilon_{i,q} \in \{0,1\}$ such that $-q_i + 2 + \varepsilon_{i,q}$ is even, and 
define 
\begin{align*}
\gamma_{i,q} &= b^{-q_i+2+\varepsilon_{i,q}} \gamma_{i,p}\pr 
& \text{if } q_i < 0 \text{ and}\\
\gamma_{i,q}  &=\gamma_{i,p}\pr & \text{if } q_i > 0;
\end{align*}
and we define $\gamma_{i,q}\pr$ similarly as $\gamma_{1,q}\pr$.
 Inductively, after $2r$ steps, we get the positive $3$-braid 
\begin{align*}
\gamma_{r,q}\pr = \Delta^{2\ell} a^{\widetilde{p_1}}b^{\widetilde{q_1}}
 \cdots a^{\widetilde{p_{r}}}b^{\widetilde{q_{r}}}
\end{align*}
with 
 \begin{align*}
 \widetilde{p_i}=
 \begin{cases}
2 + \varepsilon_{i,p} & \text{if } p_i < 0, \\
p_i  & \text{if } p_i > 0,
\end{cases} \qquad \text{and}\qquad 
\widetilde{q_i}=
 \begin{cases}
2 + \varepsilon_{i,q} & \text{if } q_i < 0, \\
q_i  & \text{if } q_i > 0, 
\end{cases} 
 \end{align*}
 for all $1 \leq i \leq r$; 
 so that $\widetilde{p_1},\widetilde{q_1},
 \dots \widetilde{p_{r}},\widetilde{q_{r}} \geq 2$. 
 By \Cref{prop:posupsilon},
we have
 \begin{align*}
 \upsilon\left( \widehat{\gamma_{r,q}\pr }\right) &= 
 -\frac{\sum\limits_{\stackrel{i=1}{ p_i > 0}}^r p_i + \sum\limits_{\stackrel{i=1}{ q_i > 0}}^r q_i + \sum\limits_{\stackrel{i=1}{ p_i < 0}}^r \left(2 + \varepsilon_{i,p}\right) + \sum\limits_{\stackrel{i=1}{ q_i < 0}}^r \left(2 + \varepsilon_{i,q}\right)}{2}+r-2\ell.
 \end{align*}
 Now, note that if $p_i < 0$ for some $1 \leq i \leq r$, then there is a cobordism of genus $1$ between $\widehat{\gamma_{i,p}\pr}$ and $T_{2,2m+1} \# \widehat{\gamma_{i-1,q}\pr }$ by using two saddle moves, where $m = \frac{-p_i+2+\varepsilon_{i,p}}{2}$,
   so similarly as in
 \eqref{eq:exintro} from \Cref{ex:introstrategy},
we have
 \begin{align*}
 \upsilon\left( \widehat{\gamma_{i-1,q}\pr }\right) \geq  \upsilon\left( \widehat{\gamma_{i,p}\pr}\right) + m -1 
 = \upsilon\left( \widehat{\gamma_{i,p}\pr}\right) + \frac{-p_i+\varepsilon_{i,p}}{2} .
 \end{align*}
 Similarly, if $q_i < 0$ for some $1 \leq i \leq r$, then 
 $
 \upsilon( \widehat{\gamma_{i,p}\pr }) \geq  \upsilon( \widehat{\gamma_{i,q}\pr}) + \frac{-q_i+\varepsilon_{i,q}}{2} .
 $
 In addition, if $p_i > 0$, then we have $\upsilon( \widehat{\gamma_{i,p}\pr})= \upsilon( \widehat{\gamma_{i-1,q}\pr})$, and
 if $q_i > 0$, then $\upsilon( \widehat{\gamma_{i,q}\pr})= \upsilon( \widehat{\gamma_{i,p}\pr})$.
 We conclude
 \begin{align*}
  \upsilon\left( \widehat{\gamma}\right) &=
 \upsilon\left( \widehat{\gamma_{0,q}\pr}\right) \geq 
 \upsilon\left( \widehat{\gamma_{r,q}\pr}\right) + \sum\limits_{\stackrel{i=1}{ p_i < 0}}^r \frac{-p_i+\varepsilon_{i,p}}{2} + \sum\limits_{\stackrel{i=1}{ q_i < 0}}^r \frac{-q_i+\varepsilon_{i,q}}{2}  \\
 &=  -\frac{\sum\limits_{\stackrel{i=1}{ p_i > 0}}^r p_i + \sum\limits_{\stackrel{i=1}{ q_i > 0}}^r q_i + \sum\limits_{\stackrel{i=1}{ p_i < 0}}^r \left( p_i+2\right) + \sum\limits_{\stackrel{i=1}{ q_i < 0}}^r \left(q_i+2\right)}{2} +r-2\ell
 \\&= -\frac{\sum\limits_{i=1}^r (p_i + q_i) }{2}+r-2\ell - \#\{ i \mid p_i < 0 \} - \#\{i \mid q_i < 0\}.\qedhere
 \end{align*}
\end{proof}

\begin{proof}[Proof of \Cref{lem:technical2}]
The strategy of the proof is the same as in the proof of \Cref{lem:technical1}. Here, we need $2r-1$ steps corresponding to the $2r-1$ exponents $p_1, q_1,\dots ,p_{r-1},q_{r-1},p_r$ of $\gamma$. The steps are similar as in the proof of \Cref{lem:technical1}, the only change is that we multiply $\gamma_{i-1,q}\pr$ by a power of $b$ if $p_i <0$, and $\gamma_{i,p}\pr$ by a power of $a$ if $q_i < 0$ (since $a\Delta^{2\ell+1}= \Delta^{2\ell+1}b$ and $b\Delta^{2\ell+1}= \Delta^{2\ell+1} a$). Thus, starting with $\gamma_{0,q}\pr=\gamma$, after $2r-1$ steps we obtain the positive $3$-braid
\begin{align*}
\gamma_{r,p}\pr = \Delta^{2\ell+1} a^{\widetilde{p_1}}b^{\widetilde{q_1}}
 \cdots a^{\widetilde{p_{r-1}}}b^{\widetilde{q_{r}-1}}a^{\widetilde{p_r}}
\end{align*}
with 
 \begin{align*}
 \widetilde{p_i}&=
 \begin{cases}
2 + \varepsilon_{i,p} & \text{if } p_i < 0, \\
p_i  & \text{if } p_i > 0,
\end{cases} \qquad \text{and} \qquad
\widetilde{q_i}=
 \begin{cases}
2 + \varepsilon_{i,q} & \text{if } q_i < 0, \\
q_i  & \text{if } q_i > 0.
\end{cases} 
 \end{align*}
 By \Cref{lem:Upsilon1_Delta2l+1}, we have
\begin{align*}
\upsilon\left(\gamma_{r,p}\pr \right)
&=
 -\frac{\sum\limits_{\stackrel{i=1}{ p_i > 0}}^{r} p_i + \sum\limits_{\stackrel{i=1}{ q_i > 0}}^{r-1} q_i + \sum\limits_{\stackrel{i=1}{ p_i < 0}}^r \left(2 + \varepsilon_{i,p}\right) + \sum\limits_{\stackrel{i=1}{ q_i < 0}}^{r-1} \left(2 + \varepsilon_{i,q}\right)}{2}+r-2\ell-\frac{3}{2}.
\end{align*}
Since the steps we performed have similar effects on $\upsilon\left(\widehat{\gamma}\right)$ as the ones in the proof of \Cref{lem:technical1}, we get
 \begin{align*}
  \upsilon\left( \widehat{\gamma}\right) &=
 \upsilon\left( \widehat{\gamma_{0,q}\pr }\right) \geq 
 \upsilon\left( \widehat{\gamma_{r,p}\pr }\right) + \sum\limits_{\stackrel{i=1}{ p_i < 0}}^r \frac{-p_i+\varepsilon_{i,p}}{2} + \sum\limits_{\stackrel{i=1}{ q_i < 0}}^{r-1} \frac{-q_i+\varepsilon_{i,q}}{2}  \\
 &=  -\frac{\sum\limits_{\stackrel{i=1}{ p_i > 0}}^r p_i + \sum\limits_{\stackrel{i=1}{ q_i > 0}}^{r-1} q_i + \sum\limits_{\stackrel{i=1}{ p_i < 0}}^r \left( p_i+2\right) + \sum\limits_{\stackrel{i=1}{ q_i < 0}}^{r-1} \left(q_i+2\right)}{2} +r-2\ell-\frac{3}{2}
 \\&=  -\frac{\sum\limits_{i=1}^{r-1} (p_i + q_i)+p_r }{2}+r-2\ell -\frac{3}{2}- \#\{ i \mid p_i < 0 \} - \#\{i \mid q_i < 0\}.\qedhere
 \end{align*}

\end{proof}

\subsection{Further discussion of \texorpdfstring{\Cref{thm:upsilon}}{Theorem 1.1}}\label{sec:discussion}

In this section, we provide some further context on our main result. In particular, in \Cref{sec:prooftechnique} we will discuss why it might be surprising that our proof strategy works for all $3$-braid knots.

\subsubsection{Comparison of upsilon and the classical signature}\label{rem:upsilonsignature}

By \Cref{thm:upsilon} and \Cref{prop:signature}, we have
\begin{align}\label{eq:signatureupsilon}
\sigma(K) = 2\upsilon(K)
\end{align}
for any knot $K$ that is the closure of a $3$-braid $\gamma = \Delta^{2\ell}a^{-p_1}b^{q_1}\cdots a^{-p_r}b^{q_r}$ for integers $\ell \in \Z$, $r\geq 1$ and $p_i,q_i\geq 1$ for $i\in \{1, \dots, r\}$.
Computations of the signature for torus knots (and links) of braid index $3$, first done by Hirzebruch, Murasugi and Shinora \cite[Proposition 9.1, pp.~34-35]{murasugibook}, together with \Cref{lem:Upsilontorus} from \Cref{sec:upsilonprops} imply that the equality in \eqref{eq:signatureupsilon} is in fact true 
 for all $3$-braid knots $K$ except for the cases that $K=\pm T_{3,3\ell+1}$ for odd $\ell >0$ or $K=\pm T_{3,3\ell+2}$ for odd $\ell >0$. In the exceptional cases, we have $\sigma(K) = 2\upsilon(K)-2$.
As mentioned in the introduction, this
improves the inequality $ \vert \upsilon(K) - \frac{\sigma(K)}{2}\vert \leq 2$ for all $3$-braid knots $K$ in \cite[Proposition 4.4]{Feller_2017}.

It was shown in \cite[Theorem 1.2]{OSSunorientedknotFloer} that $\vert \upsilon(K) - \frac{\sigma(K)}{2} \vert$
gives a lower bound on the \emph{nonorientable smooth $4$-genus} of a knot $K$, denoted $\gamma_4(K)$, the minimal first Betti number of a nonorientable 
surface in $B^4$ that meets the boundary $S^3$ along $K$. 
The similarity of the invariant $\upsilon$ and the classical signature $\sigma$ on $3$-braid knots $K$ described above clearly does not lead to a good lower bound on $\gamma_4(K)$. 

 However, the equality $\sigma(K) = 2\upsilon(K)$ for most $3$-braid knots is actually no great surprise when noting that in fact $ \vert \upsilon(K) - \frac{\sigma(K)}{2} \vert \leq 1$ must be true for all $3$-braid knots $K$ for the following reason. It is not hard to see that for every $3$-braid knot $K$, there is a nonorientable band move to a $2$-bridge knot $J$, which is alternating \cite{goodrick}. This implies that the \emph{nonorientable cobordism distance} $d_\gamma(K,J)=\gamma_4(K \# -J)$ between $K$ and $J$ is bounded from above by $1$. On the other hand, using that $\upsilon$ and $\sigma$ induce homomorphisms $\CC \to \Z$
  (see \Cref{sec:upsilonprops} and \cite{murasugi}), 
the inequality $\vert \upsilon(K) - \frac{\sigma(K)}{2} \vert \leq \gamma_4(K)$ implies that 
\begin{align*}
\left \vert \upsilon(K) - \frac{\sigma(K)}{2}\right \vert &
=\left \vert \upsilon\left(K \# -J \right) - \frac{\sigma\left(K \# -J \right)}{2} \right \vert
\leq d_\gamma(K,J) \leq 1,
\end{align*}
where we used $\upsilon(J) = \frac{\sigma(J)}{2}$ by \Cref{prop:alternating}. 

Note that a similar argument shows that $ \vert \upsilon(K) - \frac{\sigma(K)}{2} \vert \leq 2$ for all $4$-braid knots $K$, using two nonorientable band moves to transform $K$ into a $2$-bridge link, which is also alternating.

\subsubsection{On the proof technique}\label{sec:prooftechnique}

As mentioned in the introduction, it came as a surprise to the author that our proof strategy works not only for positive $3$-braid knots, but for all $3$-braid knots.
Let us make this more precise.

The proofs in \Cref{sec:upsilonpos} and \Cref{sec:upsilongeneral} imply, for any $3$-braid knot $K$, the existence of 
cobordisms $C_1$ and $C_2$ of genus $g(C_1)$ and $g(C_2)$ between $K$ and (connected sums of) torus knots $T_1$ and $T_2$, respectively, such that
\begin{align*}
g\left(C_1\right) + g\left(C_2\right) = \left \vert \upsilon(T_2)-\upsilon(T_1)\right \vert
\end{align*}
and 
\begin{align*}
\upsilon(K) = \upsilon(T_1) + g(C_1) = \upsilon(T_2)- g(C_2).
\end{align*}
For example, for knots $K$ that are closures of positive $3$-braids of Garside normal form \eqref{eq:oddpower}, the proof of \Cref{lem:Upsilon1_Delta2l+1Upper} shows the existence of such a cobordism $C_1$ for $T_1 = J_{\varepsilon}$ as in the proof of \Cref{lem:upsilonUpperBound}; and the existence of such a cobordism $C_2$ between $K$ and $T_2 = T_{3,3\left(\ell +r\right)+1} \# -T_{2,2r+1}$ follows from the proof of \Cref{lem:Upsilon1_Delta2l+1}.

The same strategy would work to determine the concordance invariants $s$ and $\tau$ for all positive $3$-braid knots $K$. Indeed, every positive $3$-braid knot can be realized as the slice of a cobordism $C$ between the unknot $U$ and a torus knot $T$ of braid index $3$ such that $g(C) = \vert \tau(U)-\tau(T)\vert=\vert s(U)-s(T)\vert$ \cite[Proposition 4.1]{lobb}. 
However, in contrast, there are $3$-braid knots where this strategy provably fails to determine $s$ and $\tau$. A concrete example is the $3$-braid knot $10_{125}$ --- the closure of $a^{-5}ba^3b$ \cite{knotinfo} --- which is not squeezed \cite[Example 3.1]{lobb}. This means that every cobordism $C$ between two connected sums of torus knots $T_1$ and $T_2$ that has $10_{125}$ as a slice satisfies $g(C)>\vert\tau(T_2)-\tau(T_1)\vert=\vert s(T_2)- s(T_1)\vert$.

\subsubsection{Comparison of the normal forms for $3$-braids}\label{sec:comparisonnormalforms}

An algorithm described in \cite[Section 7]{birman1993} as Schreier's solution to the conjugacy problem \cite{schreier} can be used to convert $3$-braids in Garside normal form (cf.~\Cref{def:garsidenormalform}) to $3$-braids in Murasugi normal form (cf.~\Cref{def:murasuginormalform}):
If $\gamma$ is a $3$-braid of Garside normal form \eqref{eq:evenpower},
then 
\begin{align*}
\gamma \sim 
\Delta^{2(\ell+r)} a^{-1} b^{p_1-2}a^{-1}b^{q_1-2}\cdots a^{-1} b^{p_r-2}a^{-1}b^{q_r-2},
\end{align*}
and if $\gamma$ is of Garside normal form \eqref{eq:oddpower},
then 
\begin{align*}
\gamma \sim \Delta^{2(\ell+r)} a^{-1} b^{p_1-2}a^{-1}b^{q_1-2}\cdots a^{-1} b^{p_{r-1}-2}a^{-1}b^{q_{r-1}-2} a^{-1}b^{p_r-2}.
\end{align*}
In addition, it is easy to see how $3$-braids of Garside normal form \eqref{eq:linkcase} or \eqref{eq:torusknotcase} 
are conjugate to braids of Murasugi normal form \eqref{eq:case1} or \eqref{eq:case2}.

\section{On alternating distances of $3$-braid knots}\label{sec:alternation}

In this section, we prove \Cref{cor:alt} from the introduction and provide lower and upper bounds on the alternation number and dealternating number of any $3$-braid knot which differ by $1$.

\subsection{Alternating distances of positive $3$-braid knots}\label{sec:altpos}

We will prove the following proposition.

\begin{prop}\label{thm:alt}
Let $K$ be a knot that is the closure of a positive $3$-braid. Then
\begin{align*}
\alt(K) &= \dalt(K)
= \tau(K) +\upsilon(K) 
\\&=
\begin{cases}
\ell & \text{if } K \text{ is the torus knot } T_{3,3\ell+k} \text{ for } \ell \geq 0, k \in \{1, 2\},\\
r+\ell-1 & \text{if } K \text{ is the closure of a braid of the form in } \eqref{eq:evenpower} \text{ or } \eqref{eq:oddpower},\\
\end{cases}
\end{align*}
where \eqref{eq:evenpower} and \eqref{eq:oddpower} refer to the Garside normal forms from \Cref{prop:normalform}.
\end{prop}

\begin{rem}
Some of the cases in \Cref{thm:alt} have already been proved by other authors. Indeed, Feller, Pohlmann and Zentner used the observation \eqref{eq:lowerboundtauups} below to show that $\alt\left(T_{3,3\ell+k}\right)=\ell$ for all $\ell \geq 0$, $k\in\{1,2\}$ \cite[Theorem 1.1]{Feller_2018}. The upper bound they used was provided by \cite[Theorem 8]{kanenobu}; in fact, the equality had already been shown by Kanenobu in half of the cases, namely when $\ell$ is even. 
Moreover, Abe and Kishimoto \cite[Theorem 3.1]{abekishimoto} showed that $\alt(K)=\dalt(K) = r + \ell -1$ if $K$ is a knot that is the closure of a positive $3$-braid of the form in \eqref{eq:evenpower}. However, to the best of this author's knowledge, 
it is new that $\alt(K)  = g(K) + \upsilon(K)$ for all positive $3$-braid knots $K$. Recall that $\tau(K) = g(K)$ for all positive $3$-braid knots $K$ by \Cref{eq:tauGamma} from \Cref{sec:braids}.
\end{rem}

Before we prove \Cref{thm:alt}, let us provide the necessary definitions and background.
The \emph{Gordian distance} $d_G(K,J)$ between two knots $K$ and $J$ is the minimal number of crossing changes needed to transform a diagram of $K$ into a diagram of $J$, where the minimum is taken over all diagrams of $K$ \cite{murakami}. 
The \emph{alternation number} $\alt(K)$ of a knot $K$ is defined as the minimal Gordian distance of the knot $K$ to the set of alternating knots \cite{kawauchi}, \ie 
\begin{align*}
\alt(K) = \text{min} \left\{d_G(K,J) \mid J \text{ is an alternating knot}\right\}.
\end{align*}
The \emph{dealternating number} $\dalt(K)$ of a knot $K$ is defined via a more diagrammatic approach \cite{adamsalmostalt}: it is the minimal number $n$ such that $K$ has a diagram that can be turned into an alternating diagram by $n$ crossing changes. 
It follows from the definitions that
\begin{align}\label{eq:altdalt}
\alt(K) \leq \dalt(K)
\end{align}
for any knot $K$ and $\alt(K)=\dalt(K) = 0$ if and only if $K$ is alternating. Note that there are families of knots for which the difference between the alternation number and the dealternating number becomes arbitrarily large \cite[Theorem 1.1]{lowrance}. 

In the proof of \Cref{thm:alt}, we will use that 
\begin{align}\label{eq:lowerboundtauups}
\left\vert \tau(K) + \upsilon(K) \right \vert \leq \alt(K)
\end{align}
for any knot $K$.
In fact, for all alternating knots $K$, we have
\begin{align}\label{eq:altinvariants}
\tau(K) = \frac{s(K)}{2}=-\upsilon(K)=-\frac{\Upsilon_K(t)}{t}=-\frac{\sigma(K)}{2}
\end{align}
for any $t \in (0,1]$ (see \cite[Theorem 1.4]{Ozsv_th_2003}, \cite[Theorem 3]{Rasmussen} and \cite[Theorem 1.14]{OSS2017}), where $s$ denotes Rasmussen's concordance invariant from Khovanov homology \cite{Rasmussen}. It follows from \cite[Theorem 2.1]{abe} --- which builds on ideas of Livingston \cite[Corollary 3]{livingston04} --- that the absolute value of the difference of any two of the invariants in \eqref{eq:altinvariants} is a lower bound on $\alt(K)$.  It was first observed in \cite{Feller_2018} that the upsilon invariant fits very well in this context (see also \cite[Lemma 8]{friedllivingstonzentner}).

Another main ingredient of our proof of \Cref{thm:alt} is the inequality
\begin{align}\label{upperbounddalt}
\operatorname{dalt}(\widehat{\gamma}) \leq r-1
\end{align}
for any positive $3$-braid $\gamma = a^{p_1}b^{q_1}\cdots a^{p_r}b^{q_r}$ with integers $r \geq 1$ and $p_i, q_i \geq 1$ for $i\in \{1, \dots, r\}$ \cite[Lemma 2.2]{abekishimoto}.

\begin{proof}[Proof of \Cref{thm:alt}]
Let $K$ be a knot that is the closure of a positive $3$-braid $\gamma$ of the form in \eqref{eq:evenpower} or \eqref{eq:oddpower}  from \Cref{prop:normalform} with $\ell \geq 0$.
We claim that then
\begin{align}\label{eq:altCDcase}
r + \ell -1 = \tau(K) + \upsilon(K) =\left\vert \tau(K) + \upsilon(K) \right \vert   \leq \alt(K) \leq \dalt(K) \leq r+\ell-1,
\end{align}
which implies the statement of the proposition for these knots. The two equalities in \eqref{eq:altCDcase} directly follow from our computations of $\upsilon(K)$ in \Cref{prop:posupsilon} and \Cref{eq:tauGamma} applied to $\gamma$.  
The first two inequalities are direct consequences of the inequalities \eqref{eq:lowerboundtauups} and \eqref{eq:altdalt}.
Finally, the last inequality follows from inequality \eqref{upperbounddalt} applied to the particular braid representatives of $K$ considered in the proof of \Cref{cor:switches2}.

For torus knots of braid index $3$, the statement follows analogously. More precisely, if $K= T_{3,3\ell+k}$ for $\ell \geq 0$ and $k \in \{1,2\}$, then by \Cref{eq:tauTorus,lem:Upsilontorus}, we have
$
\left\vert \tau(K) + \upsilon(K) \right \vert 
=\ell.
$
In addition, the inequality in \eqref{upperbounddalt} applied to the 
particular braid representatives of $K$ considered in the proof of \Cref{cor:switches2} implies that $\dalt\left(T_{3,3\ell+k}\right) \leq \ell$.
\end{proof}

From \Cref{thm:alt}, it is easy to deduce that the alternating positive $3$-braid knots are precisely the unknot and the connected sums $T_{2,2p+1} \# T_{2,2q+1}$ of two torus knots of braid index $2$ for $p, q \geq 0$. This was already known; in fact, the stronger statement is true that the only prime alternating positive braid knots are the torus knots of braid index $2$ \cite[Corollary 3]{baader}. Note that by \cite{morton} (see also \cite[Corollary 7.2]{birman1993}), the only composite $3$-braid knots are the connected sums $T_{2,2p+1} \# T_{2,2q+1}$ for $p, q \in \Z$. 

By \cite[Theorem 1.1]{abe}, the only torus knots with alternation number $1$ are the torus knots $T_{3,4}$ and $T_{3,5}$. A knot with dealternating number $1$ is called \emph{almost alternating}. 

\begin{cor}\label{cor:almostalt}
A positive $3$-braid knot is almost alternating if and only if it is one of the torus knots $T_{3,4}$ and $T_{3,5}$ or it is represented by a braid of the form
\begin{align*}
a^{p_1}b^{q_1}a^{p_2}b^{q_2},\quad
\Delta a^{p_1}b^{q_1}a^{p_2},\quad 
\Delta^2 a^{p_1} b^{q_1} \quad  \text{or}\quad 
\Delta^3 a^{p_1}
\end{align*}
for some integers $p_1, p_2, q_1, q_2 \geq 2$. 
\end{cor}

\begin{proof}
This follows directly from \Cref{thm:alt}.
\end{proof}


\begin{rem}
In particular, the seven positive $3$-braid knots with crossing number $12$ (cf. \cite{knotinfo}) are all almost alternating.
\end{rem}

\begin{rem}\label{rem:Turaevgenus}
Our results imply that the Turaev genus equals the alternation number for all positive $3$-braid knots. Indeed, let $K$ be a knot that is the closure of a positive braid of the form in \eqref{eq:evenpower} or \eqref{eq:oddpower} with $\ell \geq 0$.
Then we have
\begin{align}\label{eq:Turaevgenus}
g_T(K)&=\alt(K) = \dalt(K) =
r+\ell-1,
\end{align}
where $g_T(K)$ denotes the Turaev genus of the knot $K$. The \emph{Turaev genus} $g_T(K)$ of a knot $K$
is another alternating distance \cite{lowrance}, which 
was first defined in \cite{dasbachfuter} as the minimal genus of a Turaev surface $F(D)$, where the minimum is taken over all diagrams $D$ of $K$. The Turaev surface $F(D)$ is a closed orientable surface embedded in $S^3$ associated to the diagram $D$. 
It is formed by building the natural cobordism between the circles in the two extreme Kauffman states (the \emph{all-$A$-state} and the \emph{all-$B$-state}) of the diagram $D$ via adding saddles for each crossing of $D$, and then capping off the boundary components with disks. More details on the definition can be found \eg in a survey by Champanerkar and Kofman 
\cite{turaevsurvey}.

The equality $g_T(K)=\dalt(K)$ in \eqref{eq:Turaevgenus} 
 easily follows from \Cref{thm:alt}, the inequalities 
$\vert \tau(K) + \frac{\sigma(K)}{2} \vert \leq g_T(K)$ \cite[Theorem 1.1]{dasbachlowrance} and 
$g_T(K) \leq \dalt(K)$ \cite[Cor. 5.4]{abekishimoto},
 and the fact that $\sigma(K) = 2\upsilon(K)$ for all knots that are closures of positive braids of Garside normal form \eqref{eq:evenpower} or \eqref{eq:oddpower} (see \Cref{rem:upsilonsignature}). 
 
It is not known whether the alternation number and the Turaev genus of a knot are in general comparable: it is not known whether $\alt(K) \leq g_T(K)$ for all knots $K$ 
(see \cite[Question 3]{lowrance}). However, it was shown by Abe and Kishimoto that $ g_T\left(T_{3,3\ell+k}\right) =\dalt\left(T_{3,3\ell+k}\right) =\ell$ for all $\ell \geq 0$ and $k \in \{1,2\}$ \cite[Theorem 5.9]{abekishimoto}, so $g_T(K)=\alt(K) = \dalt(K) $ is true for all positive $3$-braid knots.
\end{rem}

\begin{rem}\label{rem:singular}
In \cite{friedllivingstonzentner}, 
Friedl, Livingston and Zentner introduce
the invariant $\mathcal{A}_s(K)$, the minimal number of double point singularities in a generically immersed concordance from a knot $K$ to an alternating knot. 
In the case that the alternating knot is the unknot, this is the well studied invariant $c_4(K)$ called the \emph{$4$-dimensional clasp number}
\cite{shibuya}. 
A sequence of crossing changes in a diagram of a knot $K$ leading to a diagram of an alternating knot $J$ realizes an immersed concordance from $K$ to $J$ where any crossing change gives rise to a double point singularity in the concordance. 
We thus have $\mathcal{A}_s(K) \leq \alt(K)$ for any knot $K$, which resembles the inequality $c_4(K) \leq u(K)$ between the $4$-dimensional clasp number and the unknotting number $u(K)$ of $K$. 
Moreover, we have
$\left \vert \upsilon(K) + \tau(K) \right \vert \leq \mathcal{A}_s(K)$ for any knot $K$ \cite[Theorem 18]{friedllivingstonzentner}, so \Cref{thm:alt} implies $\mathcal{A}_s(K) = \alt(K)$ for all positive $3$-braid knots $K$.
\end{rem}

We are now ready to prove \Cref{cor:alt} from the introduction.

\begin{proof}[{Proof of \Cref{cor:alt}}]
The corollary follows directly from \Cref{thm:alt}, \Cref{rem:Turaevgenus} and \Cref{rem:singular}.
\end{proof}

\subsection{Bounds on the alternation number of general $3$-braid knots}\label{sec:altgeneral}

In the following, we turn our attention to $3$-braid knots in general, which are not necessarily the closure of positive $3$-braids. We will use that
\begin{align}\label{eq:lowerboundupss}
\left\vert \frac{s(K)}{2} + \upsilon(K) \right \vert \leq \alt(K)
\end{align}
for any knot $K$, which follows from \cite[Theorem 2.1]{abe}, see also equation \eqref{eq:altinvariants} from \Cref{sec:altpos}. Rasmussen's invariant $s$ was computed for all $3$-braid knots in Murasugi normal form (cf.~\Cref{def:murasuginormalform}) by Greene.\footnote{These computations were generalized to all links that are closures of $3$-braids in \cite{martin2019annular}.}

\begin{cor}\label{cor:3braidalt}
Let $\gamma = \Delta^{2\ell}a^{-p_1}b^{q_1}\cdots a^{-p_r}b^{q_r}$ for some $\ell \in \Z$, $r\geq 1$ and $p_i,q_i\geq 1$ for $i\in \{1, \dots, r\}$
such that $K=\widehat{\gamma}$ is a knot. Then
\begin{align*}
\left \vert \ell \right \vert -1 \leq \alt(K) &\leq \dalt(K) \leq \left \vert \ell \right \vert \qquad \text{if } \ell \neq 0 .
\end{align*}
\end{cor}

\begin{proof}[Proof of \Cref{cor:3braidalt}]
The lower bound on the alternation number follows from the inequality \eqref{eq:lowerboundupss}, \Cref{thm:upsilon} and the values of the invariant $s$ for $K=\widehat{\gamma}$ \cite[Proposition 2.4]{greene}, namely 
\begin{align*}
s(K)&=
\begin{cases}
-\sum\limits_{i=1}^r (p_i - q_i) +6\ell - 2& \text{if } \ell > 0,\\
-\sum\limits_{i=1}^r (p_i - q_i) +6\ell + 2& \text{if } \ell < 0.\\
\end{cases}
\end{align*}
Moreover, it follows from \cite[Theorem 2.5]{abekishimoto} that $\dalt\left(\widehat{\gamma}\right) \leq \left \vert \ell \right \vert$.
\end{proof}

\begin{rem}
An alternative way to prove the upper bound on $\dalt(K)$ in \Cref{cor:3braidalt} for $\ell \geq 1$ follows from our observations in the proof of \Cref{lem:Murasugiupper}. In fact, the braid diagrams given by the braid representatives $\gamma_1$ of $K=\widehat{\gamma}$ considered in that proof can easily be transformed into alternating diagrams by $\ell$ crossing changes: it is enough to change the positive crossings corresponding to the single generators $a$ in $\gamma_1$ to negative crossings; we obtain generators $a^{-1}$ in the corresponding braid words which then correspond to alternating braid diagrams.
\end{rem}

\begin{rem}
If $K$ is represented by a $3$-braid of Garside normal form \eqref{eq:evenpower} or \eqref{eq:oddpower} (see \Cref{def:garsidenormalform}), then using the observations in \Cref{sec:comparisonnormalforms}, \Cref{cor:3braidalt} implies
\begin{align}\label{eq:altgeneral}
\left \vert r+\ell \right \vert -1 \leq \alt(K) &\leq \dalt(K) \leq \left \vert r+\ell \right \vert& \text{if } \left \vert r+\ell \right \vert > 0  \,\text{ and}\\
\alt(K) &= \dalt(K) = 0 & \text{if }  r+\ell  =0.\nonumber
\end{align}
By \Cref{thm:alt}, the lower bound in \eqref{eq:altgeneral} is sharp whenever $K$ is the closure of a positive $3$-braid of Garside normal form \eqref{eq:evenpower} or \eqref{eq:oddpower}. 
However,
there are examples 
where the upper bound in \eqref{eq:altgeneral} is sharp.
The two easiest such examples in terms of crossing number 
are the non-alternating
knots $8_{20}$ and $8_{21}$, which are represented by the $3$-braids (cf. \cite{knotinfo})
\begin{align*}
a^3b^{-1}a^{-3}b^{-1} &\sim \Delta^{-3} a^7 \qquad \text{and}\\
a^3 b a^{-2} b^2 &\sim \Delta^{-2} a^3 b^2 a^2 b^3,
\end{align*}
respectively. The lower bound on the alternation number from \eqref{eq:altgeneral} is $\left \vert r+\ell \right \vert -1 = 0$ in both cases. Indeed, by \cite[Theorem 8.6]{baldwin} both knots are quasialternating, so all the invariants from equation \eqref{eq:altinvariants} are equal \cite[Proposition 1.4]{baldwin}, \cite{manolescuozsvath}, \cite{OSS2017}.
\end{rem}

\begin{rem}
In a similar fashion as \Cref{cor:3braidalt}, the Turaev genus of all $3$-braid knots was determined up to an additive error of at most 1 by Lowrance in \cite[Proposition 4.15]{lowrancewidth} using his computation of the Khovanov width for these knots. More precisely, we have
\begin{align*}
\left \vert \ell \right \vert -1 \leq g_T(K)  \leq \left \vert \ell \right \vert\qquad \text{if } \ell \neq 0
\end{align*}
for any knot $K$ that is represented by $\gamma = \Delta^{2\ell}a^{-p_1}b^{q_1}\cdots a^{-p_r}b^{q_r}$ for some $\ell \in \Z$, $r\geq 1$ and $p_i,q_i\geq 1$ for $i\in \{1, \dots, r\}$.
\end{rem}

\section{The fractional Dehn twist coefficient of $3$-braids in Garside normal form}\label{sec:homogenization}

In this section, we compute the fractional Dehn twist coefficient of any $3$-braid in Garside normal form (cf.~\Cref{def:garsidenormalform}).\\

The \emph{fractional Dehn twist coefficient} 
is a homogeneous quasimorphism on the braid group $B_n$ that assigns to any $n$-braid $\gamma$ a rational number $\omega(\gamma)$.
Here, a \emph{quasimorphism} on a group $G$ is any map $\varphi \colon G \to \R$ such that 
\begin{align*}
\sup_{(a,b) \in G \times G} \left \vert \varphi(ab) - \varphi(a) - \varphi(b) \right \vert =\vcentcolon D_\varphi< \infty,
\end{align*}
where $D_\varphi$ is called the \emph{defect} of $\varphi$.
A quasimorphism $\varphi \colon G \to \R$ is called \emph{homogeneous} if $\varphi\left(a^k\right) = k \varphi(a)$ for all $k \in \Z$ and $a \in G$.
Any homogeneous quasimorphism is invariant under conjugation, so $\omega(\gamma)$ is invariant under the conjugacy class of $\gamma$.

The fractional Dehn twist coefficient first appeared in \cite{gabaioertel} in a different language. It can be defined for mapping classes of general surfaces with boundary, where we here view braids as mapping classes of the $n$ times punctured closed disk. Malyutin defined the fractional Dehn twist coefficient $\omega \colon B_n \to \R$, $n \geq 2$, for all braid groups and showed that its defect is $1$ if $n \geq 3$ and $0$ if $n=2$ \cite[Theorem 6.3]{malyutin}.
We refer the reader to \cite{malyutin} for a more detailed account.
 
\begin{cor}\label{cor:FDTC}
Let $\gamma$ be a $3$-braid. Then its fractional Dehn twist coefficient is
\begin{align*}
\omega(\gamma) = 
\begin{cases}
\ell & \text{if } \gamma \text{ is conjugate to a braid in } \eqref{eq:linkcase},\\
\frac{p+1}{6} +\ell & \text{if } \gamma \text{ is conjugate to a braid in } \eqref{eq:torusknotcase},\\
r +\ell & \text{if } \gamma \text{ is conjugate to a braid in } \eqref{eq:evenpower} \text{ or } \eqref{eq:oddpower}.
\end{cases}
\end{align*}
where \eqref{eq:linkcase}--\eqref{eq:oddpower} refer to the Garside normal forms from \Cref{prop:normalform}.
\end{cor}

\begin{rem}
The fractional Dehn twist coefficient was computed for $3$-braids in Murasugi normal form (cf.~\Cref{def:murasuginormalform}) in \cite[Proposition 6.6]{hubbard2021braids}.
\end{rem}

In the proof of \Cref{cor:FDTC}, we will use that the fractional Dehn twist coefficient of any $3$-braid $\gamma$ is completely determined by the writhe $\operatorname{wr}(\gamma )$ and the \emph{homogenized upsilon invariant} $\widetilde{\upsilon}$ of $\gamma$: we have
\begin{align}\label{eq:FDTChomoUpsilon}
\omega(\gamma ) 
= \widetilde{\upsilon}(\gamma) + \frac{\operatorname{wr}(\gamma )}{2} \qquad \text{\cite[Theorem 1.3]{Feller_2019}}
\end{align}
for any $3$-braid $\gamma$.
The invariant $\widetilde{\upsilon}$ is another real-valued homogeneous quasimorphism on the braid group $B_3$ which can be defined as
\begin{align*}
\widetilde{\upsilon} \colon B_3 \to \R, \qquad 
\gamma \mapsto \widetilde{\upsilon}\left(\gamma\right) = \lim_{k \to \infty} \frac{\upsilon\left(\widehat{\gamma^{6k}ab}\right)}{6k}.
\end{align*}
More generally, Brandenbursky \cite[Theorem 2.6]{brandenbursky} showed that a homogeneous quasimorphism $B_n \to \R$ can be assigned to any concordance homomorphism $\CC \to \R$ that is bounded above by a constant multiple of the $4$-genus. We refer the reader to \cite{brandenbursky} or \cite[Appendix A]{Feller_2019} for more details on homogenized concordance invariants.

\begin{prop}\label{prop:homogupsilon}
Let $\gamma$ be a $3$-braid. Then
\begin{align*}
\widetilde{\upsilon}(\gamma) = 
\begin{cases}
-\frac{p}{2} -2\ell & \text{if } \gamma \text{ is conjugate to a braid in } \eqref{eq:linkcase},\\
-\frac{p+1}{3} -2\ell & \text{if } \gamma \text{ is conjugate to a braid in } \eqref{eq:torusknotcase},\\
-\frac{\sum\limits_{i=1}^r \left(p_i + q_i\right)}{2} +r -2\ell & \text{if } \gamma \text{ is conjugate to a braid in } \eqref{eq:evenpower}, \\
-\frac{\sum\limits_{i=1}^{r-1} \left(p_i + q_i\right)+p_r}{2} +r -2\ell-\frac{3}{2} & \text{if } \gamma \text{ is conjugate to a braid in } \eqref{eq:oddpower}. 
\end{cases}
\end{align*}
\end{prop}

\begin{proof}[Proof of \Cref{prop:homogupsilon}]
We will use that $\widetilde{\upsilon}(\alpha \beta)=\widetilde{\upsilon}(\alpha) + \widetilde{\upsilon}(\beta)$ if $\alpha$ and $\beta$ commute \cite[Lemma A.1]{Feller_2019}. In particular, for any $3$-braid $\gamma$ and any $\ell \in \Z$, we have 
\begin{align}\label{eq:homogUpstwist}
\widetilde{\upsilon}\left(\Delta^{2\ell} \gamma\right) = \widetilde{\upsilon}\left(\Delta^{2\ell}\right) + \widetilde{\upsilon}(\gamma) .
\end{align}
Moreover, by the definition of $\widetilde{\upsilon}$, \Cref{lem:Upsilontorus} and the homogeneity of $\widetilde{\upsilon}$, we have 
\begin{align}\label{eq:homogfulltwist}
\widetilde{\upsilon}\left(\Delta^{2\ell}\right) = -2\ell \qquad \text{for all } \ell \in \Z.
\end{align}
We will now compute $\widetilde{\upsilon}(\gamma)$ for the positive $3$-braids $\gamma$ of the form \eqref{eq:linkcase}--\eqref{eq:oddpower}, \ie assuming $\ell \geq 0$ in \eqref{eq:linkcase}--\eqref{eq:oddpower}. The statement of \Cref{prop:homogupsilon} will then follow from \eqref{eq:homogUpstwist} and \eqref{eq:homogfulltwist}. 

First, let $\gamma = \Delta^{2\ell}a^p$ for some $\ell\geq 0, \, p \geq 0$. If $p =0$, we have $\widetilde{\upsilon}(\gamma) = -2\ell$ by \eqref{eq:homogfulltwist}.
If $p \geq 1$, we have
\begin{align*}
\gamma^{6k}ab = \Delta^{12\ell k}a^{6pk} ab \sim \Delta^{12\ell k+1}a^{6pk-1},
\end{align*}
so by \Cref{lem:Upsilon1_Delta2l+1}, for $k \geq 1$, we get
\begin{align*}
\upsilon\left(\widehat{ \gamma^{6k}ab }\right) & = -\frac{6pk-1}{2}+1-12\ell k-\frac{3}{2} = -3pk -12\ell k, \qquad \text{hence}\\
\widetilde{\upsilon}(\gamma)& = 
\lim_{k \to \infty} \frac{\upsilon\left(\widehat{ \gamma^{6k}ab }\right)}{6k}
=\lim_{k \to \infty} \frac{-3pk -12\ell k}{6k}
= -\frac{p}{2}-2\ell .
\end{align*}

Second, let $\gamma = \Delta^{2\ell}a^p b$ for some $\ell \geq 0, \, p \in \{1,2,3\}$.
We have
\begin{align*}
\gamma^{6k}ab &= \Delta^{12\ell k}\left(ab\right)^{6k} ab
= \Delta^{12\ell k +4k}ab &\text{if } p = 1,\\
\gamma^{6k}ab &
=
\Delta^{12\ell k}\left(a^2ba^2b\right)^{3k} ab
=\Delta^{12\ell k}\left(ababab\right)^{3k} ab
= \Delta^{12\ell k +6k}ab &\text{if } p = 2, \text{ and}\\
\gamma^{6k}ab &
=
\Delta^{12\ell k}\left(a^3ba^3ba^3b\right)^{2k} ab
=\Delta^{12\ell k}\left(a^2babababa^2b\right)^{2k} ab&
\\&= \Delta^{12\ell k +8k}ab &\text{if } p = 3.
\end{align*}
By \Cref{lem:Upsilontorus}, we get
\begin{align*}
\widetilde{\upsilon}(\gamma)&
= \lim_{k \to \infty} \frac{-12\ell k - (2p+2)k}{6k} = -2\ell - \frac{p+1}{3}.
\end{align*}

Third, let $\gamma = \Delta^{2\ell}a^{p_1}b^{q_1}\cdots a^{p_r}b^{q_r}  $ for some $\ell \geq 0$, $r\geq 1$, $ p_i,q_i\geq 2$, $i \in \{ 1, \dots, r\}$. Then
\begin{align*}
\gamma^{6k}ab &= \Delta^{12\ell k}\left(a^{p_1}b^{q_1}\cdots a^{p_r}b^{q_r}\right)^{6k} ab \\&\sim \Delta^{12\ell k+1}a^{p_1-1}b^{q_1}\cdots a^{p_r}b^{q_r}\left(a^{p_1}b^{q_1}\cdots a^{p_r}b^{q_r}\right)^{6k-1}\\
&\sim \Delta^{12\ell k+1}\left(b^{q_1}a^{p_2}b^{q_2}\cdots a^{p_r}b^{q_r}a^{p_1}\right)^{6k-1}b^{q_1}a^{p_2}b^{q_2}\cdots a^{p_r}b^{p_1+q_r-1},
\end{align*}
where $p_1+q_r-1 \geq 3$. By \Cref{lem:Upsilon1_Delta2l+1}, we have
\begin{align*}
\upsilon\left(\widehat{ \gamma^{6k}ab }\right) 
&
= -3k \sum_{i=1}^r (p_i+q_i)+6kr-12\ell k-1, \qquad \text{hence}\\
\widetilde{\upsilon}(\gamma)&
= -\frac{1}{2} \sum_{i=1}^r (p_i+q_i)+r-2\ell .
\end{align*}

Finally, let $\gamma = \Delta^{2\ell+1}a^{p_1}b^{q_1}\cdots a^{p_{r-1}}b^{q_{r-1}}a^{p_r}$ for some $l \geq 0$, $r\geq 1$, $p_r \geq 2$, $p_i, q_i\geq 2,i \in \{ 1, \dots, r-1\}$. Then
\begin{align*}
\gamma^{6k}ab &
= \Delta^{12\ell k}\left( \Delta a^{p_1}b^{q_1}\cdots a^{p_{r-1}}b^{q_{r-1}}a^{p_r}\right)^{6k}ab \\
&= \Delta^{12\ell k}\left( \Delta^2 b^{p_1}a^{q_1}\cdots b^{p_{r-1}}a^{q_{r-1}}b^{p_r} a^{p_1}b^{q_1}\cdots a^{p_{r-1}}b^{q_{r-1}}a^{p_r}\right)^{3k}ab \\
&=\Delta^{12\ell k+6k}\left(b^{p_1}\cdots b^{p_r} a^{p_1}\cdots a^{p_r}\right)^{3k}ab   \\
&\sim \Delta^{12\ell k+6k}a^{q_1}b^{p_2}\cdots b^{p_r} a^{p_1}\cdots a^{p_r}\left(b^{p_1}\cdots b^{p_r} a^{p_1}\cdots a^{p_r}\right)^{3k-2}\\&\qquad b^{p_1}\cdots b^{p_r} a^{p_1}\cdots a^{p_r+1}b^{p_1+1},
\end{align*}
where $p_r+1, \,p_1+1 \geq 3$. 
By \Cref{lem:Upsilon1_Delta2lCase1}, we have
\begin{align*}
\upsilon\left( \widehat{\gamma^{6k}ab }\right) &= 
-3k \left(\sum_{i=1}^{r-1} (p_i+q_i)+p_r\right)+6kr-12\ell k-9k-1, \qquad \text{hence}\\
\widetilde{\upsilon}(\gamma)&
= -\frac{1}{2} \left(\sum_{i=1}^{r-1} (p_i+q_i)+p_r\right)+r-2\ell -\frac{3}{2}.\qedhere
\end{align*}
\end{proof}

\begin{proof}[Proof of \Cref{cor:FDTC}]
This follows directly from \Cref{prop:homogupsilon}, \Cref{eq:FDTChomoUpsilon}, and a straightforward calculation of the writhe of the braids in \eqref{eq:linkcase}--\eqref{eq:oddpower}.
\end{proof}

\begin{rem}
If $\gamma$ is a $3$-braid conjugate to a braid of the form in \eqref{eq:evenpower} or \eqref{eq:oddpower} such that $\widehat{\gamma}$ is a knot, then \Cref{prop:homogupsilon} and \Cref{thm:upsilon} imply $\widetilde{\upsilon}(\gamma) = \upsilon\left(\widehat{\gamma}\right)$. If $\gamma$ additionally is a positive $3$-braid, 
then $\omega(\gamma) = r+\ell = g\left(\widehat{\gamma}\right) + \upsilon\left(\widehat{\gamma}\right) +1$ is the minimal number from \Cref{cor:switches}/\Cref{cor:switches2}.
\end{rem}

\begin{rem}
Our computation of $\omega(\gamma)$ in \Cref{cor:FDTC} together with \cite[Theorem 1.3]{Feller_2019} completely determines $\widetilde{\Upsilon(t)}(\gamma)$ for all $0\leq t \leq 1$ for any $3$-braid $\gamma$, where $\widetilde{\Upsilon(t)}(\gamma)$ is the homogenization of the invariant $\Upsilon(t) \colon \CC \to \R$, defined similarly as the homogenization $\widetilde{\upsilon}$ of $\upsilon$.
\end{rem}

\bibliographystyle{alpha}
\bibliography{bibliography}

\end{document}